\documentclass[12pt]{amsart}
\usepackage{amsmath,amsfonts,latexsym,graphicx,amssymb,url}
\usepackage{amsmath,txfonts,pifont,bbding,pxfonts,manfnt}
\usepackage{psfrag}
\usepackage{wrapfig, subfigure,graphicx}
\usepackage{hyperref}
\usepackage[active]{srcltx}
\usepackage{pdfsync}
\usepackage[none]{hyphenat}
\newcommand{\dist}{\text{dist}} 

\newcommand{\R}{{\mathbb R}} 

\theoremstyle{plain}
\newtheorem{theorem}{Theorem}[section]

\newtheorem{lemma}[theorem]{Lemma}

\newtheorem{definition}[theorem]{Definition}
\newtheorem{remark}[theorem]{Remark}
\providecommand{\bysame}{\makebox[3em]{\hrulefill}\thinspace}

\begin{document}

\setcounter{equation}{0}










\title[The Near field refractor]
{The Near field refractor}
\author[C. E. Guti\'errez and Qingbo Huang]
{Cristian E. Guti\'errez\\
 and \\
 Qingbo Huang}
\thanks{\today\\The first author was partially supported
by NSF grant DMS--1201401. \\
The second author was partially supported
by NSF grant DMS--0502045.}
\address{Department of Mathematics\\Temple University\\Philadelphia, PA 19122}
\email{gutierre@temple.edu}
\address{Department of Mathematics and Statistics\\Wright State University\\Dayton, OH 45435}

\email{qhuang@noether.math.wright.edu}
\begin{abstract}
We present an abstract method in the setting of compact metric spaces which is applied to solve a number of problems in geometric optics.
In particular, we solve the one source near field refraction problem. That is,
we construct surfaces separating two homogenous media with different refractive indices that
refract radiation emanating from the origin into a target domain contained in an $n-1$ dimensional hypersurface.
The input and output energy are prescribed. This implies the existence of lenses focusing radiation in a prescribed manner.
\end{abstract}

\maketitle

\setcounter{equation}{0}
\section{Introduction}

We present in this paper an abstract method, having general interest, that can be applied to prove existence of solutions to a number of problems in geometric optics relative to refraction. The method is formulated in the abstract setting of compact metric spaces and it is based on the ideas of concave and convex mappings and the existence of families of functions that play the role of building blocks satisfying certain properties.
The application to show existence of solutions to various problems then consists in selecting the appropriate maps and the appropriate class of building blocks related to the specific problem. 
 
A main application of this method considered in the paper is to solve the one source near field refractor problem, that is,
to show existence of surfaces that refract radiation when the output and input intensities are prescribed. 
More precisely, we
have a domain $\Omega$ in the unit sphere $S^{n-1}$ and 
a domain $D$ contained in an $n-1$ dimensional surface in $\R^{n}$; 
$D$ is referred as the target domain or screen to be illuminated.
We also have two homogeneous and isotropic media I and II with refractive indices $n_1$ and $n_2$, respectively,
and suppose that from a point $O$ surrounded by medium I,
light emanates with intensity $f(x)$
for $x\in \Omega$, and $D$ is surrounded by media II.
We are also given in $D$ a Radon measure $\mu$ and the energy conservation equation 
$\int_{\Omega}f(x)\,dx=\mu(D)$.
We prove the existence of an optical surface
$\mathcal R$ parameterized by $\mathcal R=\{\rho(x)x: x\in\overline{\Omega}\}$,
interface between media I and II,
such that all rays refracted by $\mathcal R$ into medium II
illuminate the object $D$, and the prescribed illumination intensity distribution
at $D$ is $\mu$.
Of course, some conditions on the relative position of $D$ and the set of directions in $\Omega$ are needed to illuminate $D$.
This implies that one can design a lens refracting light beams so that the 
screen $D$ is illuminated in a prescribed way.
The lens is bounded by two optical surfaces, the ``outer'' surface is $\mathcal R$ and 
the ``inner'' one is a sphere with center at the point from where the radiation emanates.

To implement the application of our abstract result to this case, we determine that the building blocks to construct the surface solution are Descartes ovals.
Descartes ovals refract all light rays
emanating from a point $O$ into a fixed point $P$;
see Figures \ref{fig:ovalskappa<1} and \ref{fig:ovalskappa>1}.
The solution of the near field refraction problem depends in a essential way of 
a novel and very delicate study of the eccentricity of the ovals and its approximation properties,
Section \ref{sec:ovals}. 
This geometry is much more complicated than the one 
needed for the solution to the far field problem solved in 
\cite{gutierrez-huang:farfieldrefractor} using mass transport.
In particular, existence of solutions for the far field refractor problem 
can be obtained directly with this abstract method, Section \ref{sec:furtherapplications}.




Geometric optics problems -far and near field- 
received recently attention because of both its applications and 
their mathematical interest and difficulty.
To put our results in perspective, we enumerate some related results.
The problem of the far field reflector has been considered by several authors,
both mathematicians and engineers.
For example, existence and uniqueness up to dilations of solutions
for the far field reflector problem were proved by Caffarelli and Oliker 
in \cite{Caffarelli-Oliker:weakantenna} and X-J. Wang in \cite{Wang:antenna},
and for nonisotropic media by Caffarelli and Q. Huang \cite{caffarelli-huang:nonisotropicreflector}. 
$C^{1}$ regularity of solutions for the far field reflector was established in 
\cite{caffarelli-gutierrez-huang:antennaannals}.
The near field reflector problem was considered by 
Kochengin and Oliker in \cite{kochengin-oliker:nearfieldreflector}, and 
in recent work
by Karakhanyan and X-J. Wang \cite{karakhanyan-wang:nearfieldreflector}.
The far field refractor problem was for the first time considered by 
the authors, and existence and uniqueness up to dilations of solutions 
were established in \cite{gutierrez-huang:farfieldrefractor}.
Refraction problems are in general more involved that reflection problems because of 
physical constraints.
A difference between near and far field problems is that the latter
can be cast in the frame of optimal transportation. Instead, near field problems are
in general not optimal transportation problems which makes their mathematical treatment more difficult.
In particular, the presence of $\rho$ in the matrix $\mathcal A$ and on the right hand side of the pde \eqref{eq:monge-amperegeneralcase}, indicates the problem is not an optimal transport problem.
We mention that some results for the near field refractor problem have been announced in \cite{gutierrez-huang:nearfielrefractorermanno6thbirth}.

The plan of the paper is the following.
Section \ref{sec:abstract setup} contains the main results in the paper where we develop the abstract method.
Section \ref{sec:snelllaw} describes the Snell law and the physical constraints of refraction. 
Section \ref{sec:ovals} contains our study of Cartesian ovals and estimates that are essential in the application to the near field problem.
In Section \ref{sec:definitionofweaksolution} we
describe the assumptions on the domains $\Omega$ and $D$ and solve the near field problem when $\kappa<1$.
Similar existence results when $\kappa>1$ are proved in Section \ref{sec:existencekappa>1}.
In Section \ref{sec:furtherapplications} we show further applications of our method: the far field refractor problem and the second boundary value problem for the Monge-Amp\`ere equation.

Finally, the pde governing the near field refractor problem is a fully nonlinear equation of Monge-Amp\`ere type whose derivation is quite complicated and it is included in Appendix, Section \ref{sec:derivationofpde}. 

\setcounter{equation}{0}
\section{Continuous mappings and measure equations}\label{sec:abstract setup}
Let $X,Y$ be compact metric spaces, and $\omega$ a Radon measure on $X$.
Let $
Y^X$ denote the class of all set-valued maps $\Phi$ from $X$ to $Y$ such that $\Phi(x)$ is single valued for a.e. $x$ with respect to the measure $\omega$.
We say that the map $\Phi\in Y^X$ is continuous at $x_0\in X$ if given $x_k\to x_0$ and $y_k\in \Phi(x_k)$, there exists a subsequence $y_{k_j}$ and $y_0\in \Phi(x_0)$ such that $y_{k_j}\to y_0$, as $j\to \infty$.
Let $C(X,Y)$ denote the class of all $\Phi\in Y^X$ such that $\Phi$ is continuous in $X$;
and let $C_s(X,Y)$ denote the class of $\Phi\in C(X,Y)$ such that $\Phi(X)=Y$. 

\begin{lemma}\label{lm:radonmeasureontarget}
Given $\Phi\in C_s(X,Y)$, the set function defined by 
\[
M_\Phi(E)=\omega \left(\Phi^{-1}(E) \right)
\]
is a Radon measure on $Y$.
\end{lemma}
\begin{proof}
The  set $\mathcal C=\left\{E\subset Y: \Phi^{-1}(E) \text{ is $\omega$-measurable} \right\}$ is a $\sigma$-algebra containing all Borel sets in $Y$ ($\Phi^{-1}(E)=\{x\in X:\Phi(x)\cap E\neq \emptyset\}$).
Indeed, $\Phi^{-1}(\emptyset)=\emptyset$, $\Phi^{-1}(Y)=X$,  $\Phi^{-1}(\cup_{i=1}^\infty E_i)=\cup_{i=1}^\infty \Phi^{-1}(E_i)$,
$\Phi^{-1}(E^c)=\left( \Phi^{-1}(E)\right)^c\cup \left(\Phi^{-1}(E)\cap \Phi^{-1}(E^c)\right)$,
and $\omega\left( \left(\Phi^{-1}(E)\cap \Phi^{-1}(E^c)\right)\right)=0$.
If $K$ is compact in $Y$, then $\Phi^{-1}(K)$ is compact in $X$. 
In fact, let $\{x_k\}\subset \Phi^{-1}(K)$, and $y_k\in \Phi(x_k)\cap K$. Since $X$ is compact and $\Phi$ is continuous, there exist subsequences $\{x_{k_j}\}$ and $\{y_{k_j}\}$ such that $x_{k_j}\to x_0$ and $y_{k_j}\to y_0$ with $y_0\in \Phi(x_0)$,
that is, $x_0\in \Phi^{-1}(K)$.
To show the $\sigma$-additivity, let $\{E_j\}$ be a disjoint sequence of sets in $\mathcal C$. Then
$M_\Phi(\cup_{j=1}^\infty E_j)=\omega(\Phi^{-1}(E_1))+\sum_{k=2}^\infty \omega\left(\Phi^{-1}(E_k)\setminus \cup_{i=1}^{k-1}\Phi^{-1}(E_i) \right)=\sum_{k=1}^\infty M_\Phi(E_k)$.
\end{proof}

Let $C(X)$ denote the set of continuous functions in $X$ with the topology of uniform convergence.

\begin{definition}\label{def:definitionofmapTcontinuous}
If $\mathcal F\subset C(X)$ and $\mathcal T:\mathcal F\to C_s(X,Y)$, we say that $\mathcal T$ is continuous at $\phi\in \mathcal F$,  if whenever $\phi_j\in \mathcal F$, $\phi_j\to \phi$ uniformly in $X$,
$x_0\in X$, and $y_j\in \mathcal T(\phi_j)(x_0)$, then there exists a subsequence $y_{j_\ell}$ such that $y_{j_\ell}\to y_0$ with $y_0
\in \mathcal T(\phi)(x_0)$.
\end{definition}
\begin{lemma}\label{lm:continuityofmeasuresabstractapproach}
If $\mathcal F\subset C(X)$, $\mathcal T:\mathcal F\to C_s(X,Y)$ is continuous at $\phi\in \mathcal F$, and $\phi_j\to \phi$ in $C(X)$ with $\phi_j\in \mathcal F$, then
$M_{\mathcal T(\phi_j)}\to M_{\mathcal T(\phi)}$ weakly. 
\end{lemma}
\begin{proof}
It is enough to show that 
\begin{equation}\label{eq:limsuplessoncompacts}
\limsup_{j\to \infty}M_{\mathcal T(\phi_j)}(K)\leq M_{\mathcal T(\phi)}(K),
\qquad \text{for all $K\subset Y$ compact,}
\end{equation}
and
\begin{equation}\label{eq:liminfbiggeronopensets}
\liminf_{j\to \infty}M_{\mathcal T(\phi_j)}(G)\geq M_{\mathcal T(\phi)}(G),
\qquad \text{for all $G\subset Y$ open.}
\end{equation}
To prove \eqref{eq:limsuplessoncompacts}, we show that $\limsup_{j\to \infty} \mathcal T(\phi_j)^{-1}(K)\subset \mathcal T(\phi)^{-1}(K)$.
Let $x_0\in \mathcal T(\phi_j)^{-1}(K)$. So there is $y_j\in \mathcal T(\phi_j)(x_0)\cap K$, and 
since $\mathcal T$ is continuous at $\phi$, there exists a subsequence $y_{j_\ell}$ such that $y_{j_\ell}\to y_0$ with $y_0
\in \mathcal T(\phi)(x_0)$. Since $y_0\in K$, we have $x_0\in \mathcal T(\phi)^{-1}(K)$.
Hence
\begin{align*}
\limsup_{j\to \infty}M_{\mathcal T(\phi_j)}(K)&=
\limsup_{j\to \infty}\omega\left(\mathcal T(\phi_j)^{-1}(K)\right)\\
&\leq
\omega\left(\limsup_{j\to \infty}\mathcal T(\phi_j)^{-1}(K)\right)
\leq
\omega\left(\mathcal T(\phi)^{-1}(K) \right)=M_{\mathcal T(\phi)}(K).
\end{align*}

To prove \eqref{eq:liminfbiggeronopensets}, we show that $\mathcal T(\phi)^{-1}(G)\setminus E_0\subset 
\liminf_{j\to \infty}\mathcal T(\phi_j)^{-1}(G)$, with $\omega(E_0)=0$.
Indeed,
\begin{align*}
\limsup_{j\to \infty}\left( \mathcal T(\phi_j)^{-1}(G)\right)^c
&\subset 
\limsup_{j\to \infty}\left[\left( \mathcal T(\phi_j)^{-1}(G)\right)^c\cup \left(\mathcal T(\phi_j)^{-1}(G)\cap \mathcal T(\phi_j)^{-1}(G^c)\right)\right]\\
&=
\limsup_{j\to \infty}\mathcal T(\phi_j)^{-1}(G^c)\\
&\subset \mathcal T(\phi)^{-1}(G^c)\qquad \text{since $G^c$ is compact because $Y$ is compact}\\
&=
\left(\mathcal T(\phi)^{-1}(G)\right)^c\cup \left[\mathcal T(\phi)^{-1}(G)\cap \mathcal T(\phi)^{-1}(G^c)\right]
:=\left(\mathcal T(\phi)^{-1}(G)\right)^c\cup E_0,
\end{align*}
with $\omega(E_0)=0$, since $\mathcal T(\phi)$ is single valued except on a set of $\omega$-measure zero.
Therefore 
\begin{align*}
\liminf_{j\to \infty}M_{\mathcal T(\phi_j)}(G)&=
\liminf_{j\to \infty}\omega\left(\mathcal T(\phi_j)^{-1}(G)\right)\\
&\geq
\omega\left(\liminf_{j\to \infty}\mathcal T(\phi_j)^{-1}(G)\right)
\geq
\omega\left(\mathcal T(\phi)^{-1}(G) \right)=M_{\mathcal T(\phi)}(G).
\end{align*}\end{proof}

\subsection{Concave case}\label{subsect:generalconcavecase}
Let
\[
C^+(X)=\{f: \text{$f$ is continuous and positive in $X$}\}.
\]
In this subsection, we consider classes $\mathcal F\subset C^+(X)$ satisfying the following condition:
\begin{enumerate}
\item[(A1)] if $f_1,f_2\in \mathcal F$, then $f_1\wedge f_2=\min\{f_1,f_2\}\in \mathcal F$. 
\end{enumerate}

%
{\it We say that the class $\mathcal F\subset C^+(X)$ is $\mathcal T$-concave
if $\mathcal F$ satisfies (A1) and there 
exists a map $\mathcal T:\mathcal F\to C_s(X,Y)$  that
is continuous at each $\phi\in \mathcal F$ and the following condition holds}
\begin{enumerate}
\item[(A2)] if $\phi_1(x_0)\leq \phi_2(x_0)$, then $\mathcal T(\phi_1)(x_0)\subset \mathcal T(\phi_1\wedge \phi_2)(x_0)$.
\end{enumerate}

We also introduce the following condition on the class $\mathcal F$:
\begin{enumerate}
\item[(A3)] For each $y_0\in Y$ there exists an interval $(\alpha_{y_0},\beta_{y_0})$
and a family of functions
$\left\{h_{t,y_0}(x)\right\}_{\alpha_{y_0}<t<\beta_{y_0}}\subset \mathcal F$ satisfying 
\begin{enumerate}
\item[(a)] $y_0\in \mathcal T(h_{t,y_0})(x)$ for all $x\in X$,
\item[(b)] $h_{t,y_0}\leq h_{s,y_0}$ for $t\leq s$,
\item[(c)] $h_{t,y_0}\to 0$ uniformly as $t\to \alpha_{y_0}$,
\item[(d)] $h_{t,y_0}$ is continuous in $C(X)$ with respect to $t$, i.e., 
$\max_{x\in X}|h_{t',y_0}(x)-h_{t,y_0}(x)|\to 0$ as $t'\to t$, for $\alpha_{y_0}<t<\beta_{y_0}$.
\end{enumerate}
\end{enumerate}

\begin{remark}\label{rmk:Mofhequalsdelta}\rm
Notice that (A3)(a) implies that $M_{\mathcal T(h_{t,y_0})}=\omega(X)\,\delta_{y_0}$ for $\alpha_{y_0}<t<\beta_{y_0}$.
Because if $E\subset Y$ is a Borel set with $y_0\in E$, then from (A3)(a), $X\subset \mathcal T(h_{t,y_0})^{-1}(y_0)\subset 
\mathcal T(h_{t,y_0})^{-1}(E)\subset X$ and so $M_{\mathcal T(h_{t,y_0})}(E)=\omega(X)$.
If $y_0\notin E$, then $y_0\in Y\setminus E$ and so $M_{\mathcal T(h_{t,y_0})}(Y\setminus E)=\omega(X)$
and then $M_{\mathcal T(h_{t,y_0})}(E)=0$.
\end{remark}

The following is the main theorem in this section. We solve a measure equation when the given measure in $Y$ is discrete.
\begin{theorem}\label{thm:abstractcasediscritecase}
Let $X,Y$ be compact metric spaces and $\omega$ is a Radon measure in $X$.
Let $p_1,\cdots ,p_N$ be distinct points in $Y$, and $g_1,\cdots ,g_N$ be positive numbers with $N\geq 2$.

Let $\mathcal F\subset C^+(X)$ and  $\mathcal T:\mathcal F\to C_s(X,Y)$ be such that 
$\mathcal F$ is $\mathcal T$-concave and (A3) holds. 

Assume that 
\begin{equation}\label{eq:conservationofenergyomega}
\omega (X)=\sum_{i=1}^N g_i,
\end{equation}
and there exists $\rho_0=\min_{1\leq i\leq N}h_{b_i^0,p_i}$ such that 
$\mathcal M_{\mathcal T(\rho_0)}(p_i)\leq g_i$ for $2\leq i \leq N$.
Then there exist $b_i\in (\alpha_{p_i},\beta_{p_i})$, $2\leq i\leq N$, 
such that the function, with $b_1=b_1^0$,
\[
\rho(x)=\min_{1\leq i \leq N}h_{b_i,p_i}(x)
\]
satisfies
\[
M_{\mathcal T(\rho)}=\sum_{i=1}^N g_i\,\delta_{p_i}.
\]
\end{theorem}
\begin{proof}
Let $b_1=b_1^0$ and define the set
\[
W=
\left\{(b_2,\cdots ,b_N):\alpha_{p_i}<b_i\leq b_i^0; M_{\mathcal T(\rho)}(p_i)\leq g_i,i=2,\cdots ,N \right\}.
\]

By the assumptions, $(b_2^0, \cdots, b_N^0)\in W$.

{\bf Step 1:} There exist constants $L$, $\epsilon_0>0$ such that for all $(b_2,\cdots ,b_N)\in W$ 
we have $b_i\geq \alpha_{p_i}+\epsilon_0$ if $\alpha_{p_i}>-\infty$, and $b_i\geq -L$
if $\alpha_{p_i}=-\infty$, for $2\leq i \leq N$.

To prove this, we first show that the measure $M_{\mathcal T(\rho)}$ is supported on $\{p_1, \cdots, p_N\}$. This follows if we show that the set 
$E=\mathcal T(\rho)^{-1}\left(Y\setminus \{p_1,\cdots ,p_N\}\right)\subset N_0$, where $N_0:=\{x\in X:\text{$\mathcal T(\rho)(x)$ is not a singleton}\}$ has $\omega$-measure zero. 
If $z_0\in E$,
then $\mathcal T(\rho)(z_0)\cap (Y\setminus \{p_1,\cdots ,p_N\})\neq \emptyset$.
Therefore there is $p\in \mathcal T(\rho)(z_0)$ with $p\neq p_i$ for $1\leq i\leq N$.
On the other hand, $\rho(z_0)=h_{b_k,p_k}(z_0)$ for some $1\leq k\leq N$.
From (A2) we then have $\mathcal T(h_{b_k,p_k})(z_0)\subset \mathcal T(\rho)(z_0)$, but $p_k\in \mathcal T(h_{b_k,p_k})(z_0)$ by (A3)(a), and so $\mathcal T(\rho)$ is not single-valued at $z_0$.

Consequently, from \eqref{eq:conservationofenergyomega} we get 
$M_{\mathcal T(\rho)}(p_1)\geq g_1>0$, and so $\omega(\mathcal T(\rho)^{-1}(p_1))>0$.
Pick $x_0\in \mathcal T(\rho)^{-1}(p_1)\setminus N_0$.
We claim that $h_{b_1,p_1}(x_0)\leq  h_{b_i,p_i}(x_0)$ for $i\geq 2$. Otherwise, there is some $i\geq 2$ such that 
$h_{b_i,p_i}(x_0)<h_{b_1,p_1}(x_0)$. Pick $j\geq 2$ such that $h_{b_j,p_j}(x_0)=\min_{2\leq i \leq N} h_{b_i,p_i}(x_0)$.
Then $\rho(x_0)=h_{b_j,p_j}(x_0)$. Hence from (A2), $\mathcal T(h_{b_j,p_j})(x_0)\subset 
\mathcal T\left( \rho\wedge h_{b_j,p_j}\right)(x_0)=\mathcal T\left( \rho\right)(x_0)=p_1$.
But from (A3)(a), $p_j\in \mathcal T(h_{b_j,p_j})(x_0)$, a contradiction and the claim is proved.
Since $h_{b_1,p_1}(x_0)\geq C>0$, we then get $h_{b_i,p_i}(x_0)\geq C>0$ for all $i\geq 2$.
From (A3)(c), $h_{b_i,p_i}(x_0)\to 0$ as $b_i\to \alpha_{p_i}$, and therefore Step 1 is proved.

{\bf Step 2:} $W$ is compact.

It is enough to show that $M_{\mathcal T(\rho)}(p_i)$ is continuous in $b'= (b_2,\cdots ,b_N)$ for each $1\leq i\leq N$.
Let $b'_m= (b_2^m,\cdots ,b_N^m)\in W$ converging to $b'_*= (b_2^*,\cdots ,b_N^*)$.
Then by (A3)(d), $\rho_m=\left(\min_{i\neq 1}h_{b_i^m,p_i}\right) \wedge h_{b_1,p_1}\to \rho^*=\left(\min_{i\neq 1}h_{b_i^*,p_i}\right) \wedge h_{b_1,p_1}$ uniformly as $m\to \infty$.
Therefore, as in the proof of Lemma \ref{lm:continuityofmeasuresabstractapproach}, we obtain
$\limsup_{m\to \infty}M_{\mathcal T(\rho_m)}(p_i)\leq M_{\mathcal T(\rho^*)}(p_i)$ for $i=1,\cdots ,N$, and
$\liminf_{m\to \infty}M_{\mathcal T(\rho_m)}(G)\geq M_{\mathcal T(\rho^*)}(G)$ for each $G$ open.
Now choose $G$ open such that $p_i\in G$ and $p_j\notin G$ for $j\neq i$. Then
$M_{\mathcal T(\rho_m)}(G)=
\omega\left(\mathcal T(\rho_m)^{-1}(p_i) \right)
+\omega\left(\mathcal T(\rho_m)^{-1}(G\setminus \{p_i\}) \right)$, but
$\omega\left(\mathcal T(\rho_m)^{-1}(G\setminus \{p_i\}) \right)\leq \omega\left(\mathcal T(\rho_m)^{-1}(Y\setminus \{p_1,\cdots ,p_N\}) \right)=0$. 
We then obtain $\liminf_{m\to \infty}M_{\mathcal T(\rho_m)}(p_i)\geq M_{\mathcal T(\rho^*)}(p_i)$ and the continuity follows.

{\bf Step 3:} Existence of solutions.

The function $z=b_{2}+\cdots +b_{N}$ attains its minimum on
$W$ at some point $(a_{2},\cdots ,a_{N})$.
We claim that $\rho(x)=\left(\min_{2\leq i \leq N} h_{a_i,p_i}\right)\wedge h_{b_1,p_1}$ is the desired solution.
Otherwise, assume for example that $M_{\mathcal T(\rho)}(p_2)<g_2$.
Let $\bar a=(a_{2}-\epsilon,a_3,\cdots ,a_{N})$ and $\bar \rho(x)=\left(\min_{2\leq i \leq N} h_{\bar a_i,p_i}\right)\wedge h_{b_1,p_1}$. By continuity $M_{\mathcal T(\bar \rho)}(p_2)<g_2$ for all $\epsilon$ sufficiently small.

For $i\geq 3$, we claim that $\mathcal T(\bar \rho)^{-1}(p_i)\subset \mathcal T(\rho)^{-1}(p_i)$ except on a set of $\omega$-measure zero. Indeed, if $x_0\in \mathcal T(\bar \rho)^{-1}(p_i)$ and $\mathcal T(\bar \rho)(x_0)$ is a single point, then $p_i=
\mathcal T(\bar \rho)(x_0)$. 
Notice that $\bar \rho(x_0)=h_{a_i,p_i}(x_0)$. Otherwise, there exists $j\neq i$, $1\leq j\leq N$, such that $\bar \rho(x_0)=h_{\bar a_j,p_j}(x_0)<h_{a_i,p_i}(x_0)$ (we set $\bar a_1=b_1$). 
Then from (A2), $\mathcal T(h_{\bar a_j,p_j})(x_0)\subset \mathcal T(\bar \rho)(x_0)$, and by (A3)(a) $p_j\in \mathcal T(h_{\bar a_j,p_j})(x_0)$ and so $p_j=p_i$, a contradiction.
From (A3) (b), $h_{\bar a_2,p_2}\leq h_{a_2,p_2}$ so $\bar \rho(x)\leq \rho(x)$, and therefore 
$\bar \rho(x_0)=h_{a_i,p_i}(x_0)=\rho(x_0)$.
Hence, and once again from (A2), 
$\mathcal T(h_{a_i,p_i})(x_0)\subset \mathcal T(\rho\wedge h_{a_i,p_i})(x_0)=
\mathcal T(\rho)(x_0)$. Thus, $p_i\in \mathcal T(\rho)(x_0)$ from (A3)(a), so $x_0\in \mathcal T(\rho)^{-1}(p_i)$, 
and the claim is proved.
We then obtain $M_{\mathcal T(\bar \rho)}(p_i)\leq M_{\mathcal T(\rho)}(p_i)\leq g_i$ for $i\geq 3$, that is, $\bar a\in W$,
a contradiction.
\end{proof}

\begin{remark}\label{rmk:existenceofrho0forb1close}\rm
The existence of the function $\rho_0$ in Theorem \ref{thm:abstractcasediscritecase} follows if one assumes  
that $b_1^0>\alpha_{p_1}$ is given and it is sufficiently close to $\alpha_{p_1}$. 
In fact, in this case we pick
$\alpha_{p_i}<\tau_i<\beta_{p_i}$ for $2\leq i \leq N$.
Since $\mathcal F\subset C^+(X)$, we have $h_{\tau_i,p_i}(x)\geq C_i>0$ for $i\geq 2$ and $x\in X$.
Then from (A3) (c), we can choose $b_1^0$ sufficiently close to $\alpha_{p_1}$ such that $h_{\tau_i,p_i}(x)\geq \min_{2\leq i\leq N}C_i\geq h_{b_1^0,p_1}(x)$.
If $b_i^0:=\tau_i$ for $2\leq i\leq N$, we then select $\rho_0(x)=\min_{1\leq i \leq N}h_{b_i^0, p_i}(x)=h_{b_1^0,p_1}(x)$ 
and so from Remark \ref{rmk:Mofhequalsdelta} $M_{\mathcal T(\rho_0)}=\omega(X)\,\delta_{p_1}$. 
Consequently, $M_{\mathcal T(\rho_0)}(p_i)=0$ for $2\leq i \leq N$.
%
\end{remark}

%
%

\begin{theorem}\label{lm:partialcomparison}
Let $\rho,\rho^*$ be two solutions as in 
Theorem \ref{thm:abstractcasediscritecase},
with $b=(b_{1},\cdots ,b_{N})$, and $b^{*}=(b_{1}^{*},\cdots , b_{N}^{*})$.  Assume that $X$ is connected and $\omega(E)>0$ for each open set $E\subset X$.
Assume in addition that condition (A3)(b) is replaced by $h_{t,y_0}<h_{s,y_0}$ for $t<s$.
\begin{enumerate}
\item[(a)]
If $b_{1}^{*}\leq b_{1}$, then $b_{i}^{*}\leq b_{i}$ for all $1\leq i \leq N$.
In particular, if $b_{1}^{*}= b_{1}$, then $b_{i}^{*}=b_{i}$ for all $1\leq i \leq N$.
\item[(b)]
If $\rho(x_0)=\rho^*(x_0)$ at some $x_0\in X$, then
$\rho=\rho^*$.
\end{enumerate}
\end{theorem}

\begin{proof}
(a)
Let $J=\{j: b_{j}<b_{j}^{*}\}$ and $I=\{i: b_{i}^*\leq b_i\}$.
Suppose by contradiction that $J\neq \emptyset$. We have $I\neq \emptyset$ since $1\in I$.
For each $j\in J$ we have $h_{b_j,p_j}(x)< h_{b_j^*,p_j}(x)$ for all $x\in X$, since $b_{j}<b_{j}^{*}$. 
And also $h_{b_i^*,p_i}(x)\leq h_{b_i,p_i}(x)$ for all $i\in I$ and all $x\in X$.

Let $Q=\{x\in X:\text{$\mathcal T(\rho^*)(x)$ is not a singleton}\}$.
From (A2) and (A3)(a), we notice that if $x\in \mathcal T(\rho^*)^{-1}(p_i)\setminus Q$, for some $1\leq i\leq N$, then $\rho^*(x)=h_{b_i^*, p_i}(x)$.

We next prove that $\text{bdy }\mathcal T(\rho^*)^{-1}(\mathcal P_J)\subset Q$, where $\mathcal P_J=\{p_j: j\in J\}$.  Indeed, let $z_0\in \text{bdy } \mathcal T(\rho^*)^{-1}(\mathcal P_J)$ and $N_{z_0}$
be an open neighborhood of $z_0$.  Then $N_{z_0}\cap (\mathcal T(\rho^*)^{-1}(\mathcal P_J))^c$ is a nonempty open set,
since $(\mathcal T(\rho^*)^{-1}(\mathcal P_J))$ is compact.  Thus, $N_{z_0}\cap (\mathcal T(\rho^*)^{-1}(\mathcal P_J))^c\setminus Q$ has a positive measure and therefore is non empty.  We then obtain $\{z_k\}$ such that
$z_k\longrightarrow z_0$ and $z_k\in (\mathcal T(\rho^*)^{-1}(\mathcal P_J))^c\setminus Q$.
So there exists $\{p_{i_k}\}$ with $p_{i_k}=\mathcal T(\rho^*)(z_k)$ and $i_k\in I$.
We may assume that $p_{i_k}=p_i$, for some $i\in I$.
Therefore, $\rho^*(z_k)=h_{b^*_i, p_i}(z_k)$.   By taking limit, $\rho^*(z_0)=h_{b^*_i, p_i}(z_0)$.
From (A2) and (A3)(a), this yields $p_i\in \mathcal T(\rho^*)(z_0)$.  Since $\mathcal T(\rho^*)(z_0)\cap \mathcal P_J\neq\emptyset$, we obtain $z_0\in Q$.

As a consequence, $\omega( (\mathcal T(\rho^*)^{-1}(\mathcal P_J))^\circ)=\sum_{j\in J}g_j>0$.

Given $x_0\in \overline{(\mathcal T(\rho^*)^{-1}(\mathcal P_J))^\circ }$, for any open neighborhood
$N_{x_0}$ of $x_0$, we have that $N_{x_0}\cap (\mathcal T(\rho^*)^{-1}(\mathcal P_J))^\circ$ is a nonempty open set.  Therefore, as in the previous argument, there exists $x_k\in (\mathcal T(\rho^*)^{-1}(\mathcal P_J))^\circ\setminus Q$ for $k\geq 1$ such that $x_k\longrightarrow x_0$.
Hence, one may assume that there exists some $p_j$ with $j\in J$ such that 
$p_j=\mathcal T(\rho^*)(x_k)$.   So $h_{b_j^*, p_j}(x_k)=\rho^*(x_k)$ and then
$h_{b_j^*, p_j}(x_0)=\rho^*(x_0)$ by taking limit.
Therefore,
$h_{b_j^*,p_j}(x_0)\leq h_{b_i^*,p_i}(x_0)$ for all $1\leq i \leq N$.
Thus, we obtain for $j\in J$ that
\[
h_{b_j,p_j}(x_0)<h_{b_j^*,p_j}(x_0)
\leq h_{b_i^*,p_i}(x_0) \leq h_{b_i,p_i}(x_0)
\quad \text{for all $i\in I$}.
\]
Hence by continuity, there exists $N_{x_0}$ a neighborhood of $x_0$ such that
\[
h_{b_j,p_j}(y)< h_{b_i,p_i}(y)
\quad \text{for all $i\in I$, $j\in J$, and $y\in N_{x_0}$}.
\]
By definition of $\rho$ this implies that $\rho(y)=\min_{j\in J}h_{b_j,p_j}(y)$ for all $y\in N_{x_0}$.
Therefore for each $y\in N_{x_0}$ there exists $j_0\in J$, depending on $y$, such that $\rho(y)=h_{b_{j_0},p_{j_0}}(y)$.
Hence, once again by (A2) and (A3)(a), $p_{j_0}\in \mathcal T(h_{b_{j_0},p_{j_0}})(y)\subset \mathcal T(\rho)(y)$.
That is, $y\in \mathcal T(\rho)^{-1}(p_{j_0})$, and 
therefore
\[
N_{x_0}\subset \mathcal T(\rho)^{-1}\left(\mathcal P_J\right).
\]

We then have that every point $x\in 
\overline{(\mathcal T(\rho^*)^{-1}(\mathcal P_J))^\circ}$ has a neighborhood
contained in $\mathcal T(\rho)^{-1}\left(\mathcal P_J\right)$, that is,
\[
\overline{(\mathcal T(\rho^*)^{-1}(\mathcal P_J))^\circ}
\subset 
\left(\mathcal T(\rho)^{-1}(\mathcal P_J)\right)^{\circ}
\neq X.
\]
This is a contradiction with the fact that
\[
\omega\left(\mathcal T(\rho)^{-1}\left(\mathcal P_J\right)\right)
=\sum_{j\in J}g_{j}
=
\omega\left(\overline{(\mathcal T(\rho^*)^{-1}\left(\mathcal P_J\right))^\circ} \right).
\]

(b) 
If $b_1=b_1^*$, then $b_j=b_j^*$ for all $j>1$ by part (a), and we are done.
We claim that if $b_1>b_1^*$, then $b_j> b_j^*$ for all $j>1$.
Indeed, if $b_j=b_j^*$ for some $j\neq 1$,
then $b_k=b_k^*$ for all $k\neq j$ by part (a), a contradiction.
Therefore
$\rho^*(x_0)=\min h_{p_i,b_i^*}(x_0)< \min h_{p_i,b_i}(x_0)=\rho(x_0)$, a contradiction.

\end{proof}

\begin{theorem}\label{thm:uniquenessabstractcase}
Let $\{\mu_l\}$ be a sequence of discrete Radon measures in $Y$ such that
$\mu_l \longrightarrow \mu$ weakly and $\mu_l(Y)=\omega(X)$ for $l\geq 1$.
Let $\rho_l$ be a solution obtained in Theorem \ref{thm:abstractcasediscritecase} corresponding to $\mu_l$.
Assume that there exists $R_0>0$ such that $R_0\in \text{Range }(\rho_l)$ for $l\geq 1$.
Suppose that
\begin{enumerate}
\item[(i)]
For each $R_1>0$ with $R_1\in Range(h_{t,y})$, there exists $C_{R_1}>0$ such that
$C_{R_1}^{-1} \leq h_{t,y} \leq C_{R_1}$.
\item[(ii)]
For any $C_1>C_0>0$, the family $\{f\in\mathcal F: C_0\leq f \leq C_1 \text{ in $X$}\}$
is compact in $C(X)$.
\end{enumerate}
Then there exists $\rho\in \mathcal F$ satisfying $\mathcal M_{\mathcal T(\rho)}=\mu$.
\end{theorem}

\begin{proof}
By (ii) and Lemma \ref{lm:continuityofmeasuresabstractapproach}, it suffices to show $\{\rho_l\}$ is bounded from below and above.
Assume $\rho_l(x_l)=R_0$ for some $x_l\in X$.   Then there exists $h_{b_l,y_l}$ such that
$\rho_l\leq h_{b_l,y_l}$ and $R_0=\rho_l(x_l)= h_{b_l,y_l}(x_l)$.  By (i), 
$C_{R_0}^{-1} \leq h_{b_l,y_l} \leq C_{R_0}$ for some $C_{R_0}$.  Therefore, $\rho_l\leq C_{R_0}$.
To get a lower bound, given $x_1\in X$, there exists $h_{b'_l,y'_l}$ such that
$\rho_l\leq h_{b'_l,y'_l}$ and $\rho_l(x_1)= h_{b'_l,y'_l}(x_1)$.   Hence, $R_0\leq 
h_{b'_l,y'_l}(x_l)$.   Since $h_{t, y'_l}$ is continuous and decreasing to zero ((A3)(b) and (c)),
there exists $b"_l \leq b'_l$ with $R_0=h_{b"_l,y'_l}(x_l)$.
It follows from (A3)(b) that $\rho_l(x_1)\geq  h_{b"_l,y'_l}(x_1) \geq C_{R_0}^{-1}$.
Hence $\rho_l \geq C_{R_0}^{-1}$.
\end{proof}

\subsection{Convex case}\label{subsect:convexcaseBIS}
We assume here that $\mathcal F\subset C^+(X)$ and condition (A1) above is replaced by
\begin{enumerate}
\item[(A1')] if $f_1,f_2\in \mathcal F$, then $f_1\vee f_2=\max\{f_1,f_2\}\in \mathcal F$.
\end{enumerate}
{\it We say that the class $\mathcal F\subset C^+(X)$ is $\mathcal T$-convex
if $\mathcal F$ satisfies (A1') and there 
exists a map $\mathcal T:\mathcal F\to C_s(X,Y)$  that
is continuous at each $\phi\in \mathcal F$ and the following condition holds}
\begin{enumerate}
\item[(A2')] if $\phi_1(x_0)\geq \phi_2(x_0)$, then $\mathcal T(\phi_1)(x_0)\subset \mathcal T(\phi_1\vee \phi_2)(x_0)$.
\end{enumerate}
Here we substitute condition (A3) by
\begin{enumerate}
\item[(A3')] For each $y_0\in Y$ there exists an interval $(\alpha_{y_0},\beta_{y_0})$
and a family of functions
$\left\{h_{t,y_0}(x)\right\}_{\alpha_{y_0}<t<\beta_{y_0}}\subset \mathcal F$ satisfying 
\begin{enumerate}
\item[(a)] $y_0\in \mathcal T(h_{t,y_0})(x)$ for all $x\in X$,
\item[(b)] $h_{t,y_0}\geq h_{s,y_0}$ for $t\leq s$,
\item[(c)] $h_{t,y_0}\to 0$ uniformly as $t\to \beta_{y_0}$,
\item[(d)] $h_{t,y_0}$ is continuous in $C(X)$ with respect to $t$, i.e., 
$\max_{x\in X}|h_{t',y_0}(x)-h_{t,y_0}(x)|\to 0$ as $t'\to t$, for $\alpha_{y_0}<t<\beta_{y_0}$.
\end{enumerate}
\end{enumerate}
Under these assumptions we prove the following theorem.

\begin{theorem}\label{thm:abstractcasediscritecaseconvexcaseBIS}
Let $X,Y$ be compact metric spaces and $\omega$ is a Radon measure in $X$.
Let $p_1,\cdots ,p_N$ be distinct points in $Y$, and $g_1,\cdots ,g_N$ be positive numbers with $N\geq 2$.

Let $\mathcal F\subset C^+(X)$ 
and  $\mathcal T:\mathcal F\to C_s(X,Y)$ be such that 
$\mathcal F$ is $\mathcal T$-convex and (A3') holds.

Assume that 
\begin{equation}\label{eq:conservationofenergyomegabisBIS}
\omega (X)=\sum_{i=1}^N g_i.
\end{equation}
Suppose that there exists $(b^0_1, \cdots, b^0_N)$ with $h_{b_1^0,p_1}(x)\le \min_{2\le i\le N}h_{b_i^0,p_i}(x)$ on $X$.
Then there exist $b_i\in (\alpha_{p_i},\beta_{p_i})$, $2\leq i\leq N$, 
such that the function, with $b_1=b_1^0$,
\[
\rho(x)=\max_{1\leq i \leq N}h_{b_i,p_i}(x)
\]
satisfies
\[
M_{\mathcal T(\rho)}=\sum_{i=1}^N g_i\,\delta_{p_i}.
\]
\end{theorem}

\begin{proof}
Let $b_1=b_1^0$ and $\eta=\min_Xh_{b_1,p_1}>0$. From (A3')(c), there exists $\tau>0$ such that  
\[
\max_{X} h_{\beta_{p_i}-\tau,p_i}\leq \eta
\]
for all $2\leq i \leq N$.
Therefore, if $\rho_\tau=h_{b_1,p_1}\vee \left( \max_{2\leq i\leq N}h_{\beta_{p_i}-\tau,p_i}\right)=h_{b_1,p_1}$, 
then $M_{\mathcal T(\rho_\tau)}=\omega(X)\,\delta_{p_1}$, and so $M_{\mathcal T(\rho_\tau)}(p_i)=0$ for $2\leq i\leq N$.

Consider the set
\[
W(b_1)=
\left\{(b_2,\cdots ,b_N):\alpha_{p_i}< b_i\leq \beta_{p_i}-\tau ; M_{\mathcal T(\rho)}(p_i)\leq g_i,i=2,\cdots ,N \right\}.
\]
$W(b_1)\neq \emptyset$, because $(\beta_{p_2}-\tau,\cdots ,\beta_{p_N}-\tau)\in W(b_1)$.

We claim that $b_i\geq b_i^0$, for $2\leq i \le N$, for all $(b_2,\cdots ,b_N)\in W(b_1)$. 

To prove the claim, 
we first show $M_{\mathcal T(\rho)}\left(Y\setminus \{p_1,\cdots ,p_N\}\right)=0$; $\rho=\max_{1\leq i \leq N}h_{b_i,p_i}$. Indeed, for 
$z_0\in E=\mathcal T(\rho)^{-1}\left(Y\setminus \{p_1,\cdots ,p_N\}\right)$,  
there is $p\in \mathcal T(\rho)(z_0)$ with $p\neq p_i$ for $1\leq i\leq N$.
On the other hand, $\rho(z_0)=h_{b_k,p_k}(z_0)$ for some $k$.
From (A2') and (A3')(a) we then have $p_k\in \mathcal T(h_{b_k,p_k})(z_0)\subset \mathcal T(\rho)(z_0)$. So $\mathcal T(\rho)$ is not single-valued at $z_0$
and so $\omega(E)=0$.
Consequently, from \eqref{eq:conservationofenergyomegabisBIS} we get $M_{\mathcal T(\rho)}(p_1)\geq g_1>0$.

Now suppose by contradiction that $b_i<b_i^0$ for some $2\leq i \leq N$.  Since $b_1=b_1^0$, it follows from the assumption and (A3')(b) that $h_{b_1, p_1} \leq h_{b_i^0, p_i} \leq h_{b_i, p_i}$ in $X$.
This implies that for each $x_0\in X$, there is $j\neq 1$ such that $\rho(x_0)=h_{b_j, p_j}(x_0)$. By (A2'), 
$p_j\in \mathcal T(h_{b_j,p_j})(x_0)\subset \mathcal T(\rho)(x_0)$.   Thus, $X=\mathcal T(\rho)^{-1}(\{p_2,\cdots ,p_N\})$ and
$M_{\mathcal T(\rho)}(p_1)=\omega(\mathcal T(\rho)^{-1}(p_1))=0$, a contradiction and the claim is proved.

As in the proof of Theorem 2.5, one can show
that $M_{\mathcal T(\rho)}(p_i)$ is continuous in $b'= (b_2,\cdots ,b_N)$ for each $1\leq i\leq N$.
Therefore, $W(b_1)$ is compact.

Finally, to get the existence of solutions,  
the function $z=b_{2}+\cdots +b_{N}$ attain its minimum on $W(b_1)$ at some point $(a_{2},\cdots ,a_{N})$.
We claim that $\rho(x)=\left(\max_{2\leq i \leq N} h_{a_i,p_i}\right)\vee h_{b_1,p_1}$ is the desired solution.
Otherwise, we may assume, for example, that $M_{\mathcal T(\rho)}(p_2)<g_2$.
Let $\bar a=(a_{2}-\epsilon,a_3,\cdots ,a_{N})$ and $\bar \rho(x)=\left(\max_{2\leq i \leq N} h_{\bar a_i,p_i}\right)\vee h_{b_1,p_1}$. 
By continuity $M_{\mathcal T(\bar \rho)}(p_2)<g_2$ for all $\epsilon$ sufficiently small.

On the other hand, for $i\geq 3$, we claim that $\mathcal T(\bar \rho)^{-1}(p_i)\subset \mathcal T(\rho)^{-1}(p_i)$ except on a set of $\omega$-measure zero. 
Indeed, if $x_0\in \mathcal T(\bar \rho)^{-1}(p_i)$ and $\mathcal T(\bar \rho)(x_0)$ is a single point, then $p_i=\mathcal T(\bar \rho)(x_0)$. 
Notice that $\bar \rho(x_0)=h_{a_i,p_i}(x_0)$. Otherwise, if $\bar \rho(x_0)=h_{\bar a_j,p_j}(x_0)>h_{a_i,p_i}(x_0)$ for some $j\neq i$ (we set $\bar a_1=b_1$), 
then by (A2') and (A3')(a) $p_j\in \mathcal T(h_{\bar a_j,p_j})(x_0)\subset \mathcal T(\bar\rho)(x_0)$ and so $p_j=p_i$, a contradiction. Set $\bar a_2=a_2-\epsilon$.
From (A3') (b), $h_{\bar a_2,p_2}\geq h_{a_2,p_2}$ and hence $h_{a_i,p_i}(x_0)=\rho(x_0)$.
Hence, and once again from (A2'), 
$\mathcal T(h_{a_i,p_i})(x_0)\subset \mathcal T(\rho\vee h_{a_i,p_i})(x_0)=
\mathcal T(\rho)(x_0)$. Thus, $p_i\in \mathcal T(\rho)(x_0)$ from (A3')(a); so $x_0\in \mathcal T(\rho)^{-1}(p_i)$, 
and the claim is proved.
We then obtain $M_{\mathcal T(\bar \rho)}(p_i)\leq M_{\mathcal T(\rho)}(p_i)\leq g_i$ for $i\geq 3$, that is, $\bar a\in W$,
a contradiction.
\end{proof}

Similar to Theorem \ref{lm:partialcomparison} we have the following.

\begin{theorem}\label{lm:partialcomparisonconvex}
Let $\rho,\rho^*$ be two solutions as in 
Theorem \ref{thm:abstractcasediscritecaseconvexcaseBIS},
with $b=(b_{1},\cdots ,b_{N})$, and $b^{*}=(b_{1}^{*},\cdots , b_{N}^{*})$.  Assume that $X$ is connected and $\omega(E)>0$ for each open set $E\subset X$.
Assume in addition that condition (A3')(b) is replaced by $h_{s,y_0}<h_{t,y_0}$ for $t<s$.
\begin{enumerate}
\item[(a)]
If $b_{1}^{*}\leq b_{1}$, then $b_{i}^{*}\leq b_{i}$ for all $1\leq i \leq N$.
In particular, if $b_{1}^{*}= b_{1}$, then $b_{i}^{*}=b_{i}$ for all $1\leq i \leq N$.
\item[(b)]
If $\rho(x_0)=\rho^*(x_0)$ at some $x_0\in X$, then
$\rho=\rho^*$.
\end{enumerate}
\end{theorem}


\begin{theorem}\label{thm:general measure}
Let $\{\mu_l\}$ be a sequence of discrete Radon measures in $Y$ such that
$\mu_l \longrightarrow \mu$ weakly and $\mu_l(Y)=\omega(X)$ for $l\geq 1$.
Let $\rho_l$ be a solution obtained in Theorem \ref{thm:abstractcasediscritecaseconvexcaseBIS} corresponding to $\mu_l$.
Assume that there exists $R_0>0$ such that $R_0\in \text{Range }(\rho_l)$ for $l\geq 1$ and $R_0<\lim_{t\to \alpha_y^+} h_{t,y}(x)$ for $x\in X$, $y\in Y$.
Suppose that
\begin{enumerate}
\item[(i)]
For each $R_1>0$ with $R_1\in Range(h_{t,y})$, there exists $C_{R_1}>0$ such that
$C_{R_1}^{-1} \leq h_{t,y} \leq C_{R_1}$.
\item[(ii)]
For any $C_1>C_0>0$, the family $\{f\in\mathcal F: C_0\leq f \leq C_1 \text{ in $X$}\}$
is compact in $C(X)$.
\end{enumerate}
Then there exists $\rho\in \mathcal F$ satisfying $\mathcal M_{\mathcal T(\rho)}=\mu$.
\end{theorem}

\begin{proof} 
By (ii) and Lemma \ref{lm:continuityofmeasuresabstractapproach}, it suffices to show $\{\rho_l\}$ is bounded from below and above.
Assume $\rho_l(x_l)=R_0$ for some $x_l\in X$.   Then there exists $h_{b_l,y_l}$ such that
$\rho_l\geq h_{b_l,y_l}$ and $R_0=\rho_l(x_l)= h_{b_l,y_l}(x_l)$.  By (i), 
$C_{R_0}^{-1} \leq h_{b_l,y_l} \leq C_{R_0}$ for some $C_{R_0}$.  Therefore, $\rho_l\geq C_{R_0}^{-1}$.
To get an upper bound, given $x_1\in X$, there exists $h_{b'_l,y'_l}$ such that
$\rho_l\geq h_{b'_l,y'_l}$ and $\rho_l(x_1)= h_{b'_l,y'_l}(x_1)$.   Hence, $R_0\geq 
h_{b'_l,y'_l}(x_l)$.   Since $R_0<\lim_{t\to \alpha_{y'_l}^+} h_{t,y'_l}(x_l)$, 
there exists $b_l'' \leq b'_l$ with $R_0=h_{b_l'',y'_l}(x_l)$.
It follows from (A3')(b) that $\rho_l(x_1)\leq  h_{b"_l,y'_l}(x_1) \leq C_{R_0}$.
Hence $\rho_l \leq C_{R_0}$.
\end{proof}

\subsection{Convex case infinity}\label{subsect:convexcase}
We assume here that the class $\mathcal F\subset C(X)$ is $\mathcal T$-convex, i.e., 
$\mathcal F$ satisfies (A1') and there 
exists a map $\mathcal T:\mathcal F\to C_s(X,Y)$  that
is continuous at each $\phi\in \mathcal F$ and (A2') holds.
We also consider the following condition
\begin{enumerate}
\item[(A3'')] For each $y_0\in Y$ there exists an interval $(\alpha_{y_0},\beta_{y_0})$
and a family of functions
$\left\{h_{t,y_0}(x)\right\}_{\alpha_{y_0}<t<\beta_{y_0}}\subset \mathcal F$ satisfying 
\begin{enumerate}
\item[(a)] $y_0\in \mathcal T(h_{t,y_0})(x)$ for all $x\in X$,
\item[(b)] $h_{t,y_0}\leq h_{s,y_0}$ for $t\leq s$,
\item[(c)] $h_{t,y_0}\to +\infty$ uniformly as $t\to \beta_{y_0}$,
\item[(d)] $h_{t,y_0}$ is continuous in $C(X)$ with respect to $t$, i.e., 
$\max_{x\in X}|h_{t',y_0}(x)-h_{t,y_0}(x)|\to 0$ as $t'\to t$, for $\alpha_{y_0}<t<\beta_{y_0}$.
\end{enumerate}
\end{enumerate}
Under these assumptions we prove the following theorem.
\begin{theorem}\label{thm:abstractcasediscritecaseconvexcase}
Let $X,Y$ be compact metric spaces and $\omega$ be a Radon measure in $X$.
Let $p_1,\cdots ,p_N$ be distinct points in $Y$, and $g_1,\cdots ,g_N$ be positive numbers with $N\geq 2$.

Let $\mathcal F\subset C(X)$ 
and  $\mathcal T:\mathcal F\to C_s(X,Y)$ be such that 
$\mathcal F$ is $\mathcal T$-convex and (A3'') holds.

Assume that 
\begin{equation}\label{eq:conservationofenergyomegabis}
\omega (X)=\sum_{i=1}^N g_i,
\end{equation}
and there exists $\rho_0=\max_{1\leq i\leq N}h_{b_i^0,p_i}$ such that 
$\mathcal M_{\mathcal T(\rho_0)}(p_i)\leq g_i$ for $2\leq i \leq N$.
Then there exist $b_i\in (\alpha_{p_i},\beta_{p_i})$, $2\leq i\leq N$, 
such that the function, with $b_1=b_1^0$,
\[
\rho(x)=\max_{1\leq i \leq N}h_{b_i,p_i}(x)
\]
satisfies
\[
M_{\mathcal T(\rho)}=\sum_{i=1}^N g_i\,\delta_{p_i}.
\]
\end{theorem}

\begin{proof}
We convert this case to the concave case considered in Subsection \ref{subsect:convexcaseBIS}, 
and use Theorem \ref{thm:abstractcasediscritecase} to prove the theorem.
Consider the family $\mathcal F^*=\{e^{-f}: f\in \mathcal F\}\subset C^+(X)$ and the mapping $\mathcal T^*: \mathcal F^*\to C_s(X,Y)$
given by $\mathcal T^*(e^{-f})=\mathcal T(f)$.  It is easy to verify that (A1) and (A2) hold and hence $\mathcal F^*$ is $\mathcal T^*$-concave.
To verify (A3), for $y_0\in Y$, consider the interval $(-\beta_{y_0}, -\alpha_{y_0})$ and $h^*_{t, y_0}(x)=e^{-h_{-t, y_0}(x)}$.
Obviously, $\{h^*_{t, y_0}(x)\}_{-\beta_{y_0}<t<-\alpha_{y_0}} \subset \mathcal F^*$ and satisfies (A3)(a)-(d).
Set $\rho^*_0=e^{-\rho_0}=\min_{1\leq i \leq N} h^*_{-b^0_i, p_i}$. By definition of $\mathcal T^*$, 
$\mathcal T^*(\rho^*_0)=\mathcal T(\rho_0)$, and consequently $\mathcal M_{\mathcal T^*(\rho^*_0)}(p_i)= \mathcal M_{\mathcal T(\rho_0)}(p_i)\leq g_i$.
By Theorem \ref{thm:abstractcasediscritecase}, there exists $\rho^*(x)=\min_{1\leq i \leq N} h^*_{-b_i, p_i}$ satisfying the equation
$\mathcal M_{\mathcal T^*(\rho)}=\sum_{i=1}^N g_i \delta_{p_i}$.  Since $\rho^*=e^{-\rho}$, $\mathcal M_{\mathcal T(\rho)}=\sum_{i=1}^N g_i \delta_{p_i}$. 

\end{proof}

\setcounter{equation}{0}
\section{Snell's law of refraction}\label{sec:snelllaw}
Suppose $\Gamma$ is a surface in $\R^n$ that separates two media
I and II that are homogeneous and isotropic. Let $v_1$ and
$v_2$ be the velocities of propagation of light in the media I and
II respectively. The index of refraction of medium I is 
$n_1=c/v_1$, where $c$ is the velocity of propagation
of light in the vacuum, and similarly $n_2=c/v_2$. If a
ray of light\footnote{Since the refraction angle depends on the frequency of the radiation, we assume our light ray is monochromatic.} having direction $x\in S^{n-1}$ and traveling
through medium I hits $\Gamma$ at the point $P$,
and $\nu$ is the unit
normal to $\Gamma$ at $P$ going towards medium II, then this ray
is refracted in the direction $m\in S^{n-1}$ through medium II
according with the Snell law in vector form:
the vectors $x,\nu$ and $m$ are all coplanar, and the
vector $n_2m -n_1 x$ is parallel to the normal vector $\nu$, that is,
setting $\kappa=n_2/n_1$, we have
\begin{equation}\label{eq:snellvectorform}
x-\kappa \,m =\lambda \nu,
\end{equation}
for some $\lambda\in \R$. 
Making the vector product of this equation with the normal $\nu$ we obtain the
well known form of the Snell law:
$n_1\sin \theta_1= n_2\sin
\theta_2$, where $\theta_1$ is the angle between $x$ and $\nu$
(the angle of incidence),
$\theta_2$ the angle between $m$ and $\nu$ (the angle of refraction).

When $\kappa<1$, or equivalently $v_1<v_2$,
waves propagate in medium II faster than in medium I,
or equivalently, medium I is denser than medium II.
In this case the refracted rays tend to bent away from the normal,
that is the case
for example, when medium I is glass and medium II is air.
In case $\kappa>1$,
waves propagate in medium I faster than in medium II,
and the refracted rays tend to bent towards the normal.

We summarize the physical constraints of
refraction in the following lemma whose proof is in \cite{gutierrez-huang:farfieldrefractor}.
\begin{lemma}
Let $n_1$ and $n_2$ be the indices of refraction of two media I and II,
respectively, and $\kappa=n_2/n_1$.
Then a light ray in medium I with direction
$x\in S^{n-1}$ is refracted by some surface into a light ray with direction
$m\in S^{n-1}$ in medium II
if and only if $m\cdot x\ge \kappa$, when $\kappa<1$; and
if and only if $m\cdot x\ge 1/\kappa$, when $\kappa>1$.

\end{lemma}

\setcounter{equation}{0}
\section{Cartesian Ovals}\label{sec:ovals}
To resolve our problem it is important to solve first the following simpler problem:
given a point $O$ inside medium I and a point $P$ inside medium II, find an interface surface $\mathcal S$ between media I and II that refracts all rays emanating from the point $O$ into the point $P$.
Suppose $O$ is the origin, and let $X(t)$ be a curve on $\mathcal S$. 
By the Snell law of refraction the tangent vector $X'(t)$ satisfies
\[
X'(t)\cdot \left( \dfrac{X(t)}{|X(t)|} - \kappa \dfrac{P-X(t)}
{|P-X(t)|}\right)
=0.
\]
That is,
\[
|X(t)|' + \kappa  |P-X(t)|'=0.
\]
Therefore $\mathcal S$ is the Cartesian oval
\begin{equation}\label{eq:oval}
|X|+\kappa |X-P|= b.
\end{equation}
Since $f(X)=|X|+\kappa |X-P|$ is a convex function, the oval is a convex set.

In our treatment of the problem, we need to analyze the polar equation and find the refracting piece for the oval. Write $X=\rho(x)x$ with $x\in S^{n-1}$. Then 
writing $\kappa |\rho(x)x-P|=b-\rho(x)$, squaring this quantity and solving the quadratic equation yields
\begin{equation}\label{eq:definitionofrho}
\rho(x)=\dfrac{(b-\kappa^{2} x\cdot P)\pm 
\sqrt{(b- \kappa^{2}x\cdot P)^{2}- (1-\kappa^{2})(b^{2}-\kappa^{2}|P|^{2})}}{1-\kappa^{2}}.
\end{equation}
Set 
\begin{equation}\label{eq:definitionofDelta}
\Delta(t)=(b- \kappa^{2}t)^{2}- (1-\kappa^{2})(b^{2}-\kappa^{2}|P|^{2}).
\end{equation}
\subsection{Case $0<\kappa <1$}
We have
\begin{equation}\label{eq:lowerboundofdelta}
\Delta(x\cdot P)> \kappa^{2} (x\cdot P -b)^{2},
\qquad \text{if }|x\cdot P|<|P|.
\end{equation}
If $b\geq |P|$, then $O$ and $P$ are inside or on the oval, and so the oval cannot refract
rays to $P$.
If the oval is non empty, then $\kappa |P|\leq b$. In case $\kappa |P|= b$, the oval 
reduces to the point $O$. 
The only interesting case is then $\kappa |P|<b<|P|$. 
From the equation of the oval we get that $\rho(x)\leq b$.
So we now should decide which values $\pm$ to take in the definition of $\rho(x)$.
Let $\rho_{+}$ and $\rho_{-}$ be the corresponding $\rho$'s.
We claim that $\rho_{+}(x)> b$ and $\rho_{-}(x)\leq b$.
Indeed,
\begin{align*}
\rho_{+}(x)&= \dfrac{(b-\kappa^{2} x\cdot P)+ 
\sqrt{ \Delta(x\cdot P)}}{1-\kappa^{2}}\\
&\geq 
\dfrac{(b-\kappa^{2} x\cdot P)+ \kappa |b-x\cdot P|}{1-\kappa^{2}}\\
&=
b+\dfrac{\kappa^{2}(b-x\cdot P)+ \kappa |b-x\cdot P|}{1-\kappa^{2}}\\
&\geq b.
\end{align*}
The equality $\rho_+(x)=b$ holds only if $|x\cdot P|=|P|$ and $b=x\cdot P$.
So $\rho_{+}(x)>b$ if $\kappa |P|<b<|P|$.
Similarly,
\begin{align*}
\rho_{-}(x)&= \dfrac{(b-\kappa^{2} x\cdot P)-
\sqrt{ \Delta(x\cdot P)}}{1-\kappa^{2}}\\
&\leq 
\dfrac{(b-\kappa^{2} x\cdot P)- \kappa |b-x\cdot P|}{1-\kappa^{2}}\\
&=
b+\dfrac{\kappa^{2}(b-x\cdot P)- \kappa |b-x\cdot P|}{1-\kappa^{2}}\\
&\leq b.
\end{align*}
So the claim is proved.
Therefore the polar equation of the oval is then given by
\begin{equation}\label{eq:polareqovalk<1}
h(x,P,b)= \rho_{-}(x)= \dfrac{(b-\kappa^{2} x\cdot P)-
\sqrt{ \Delta(x\cdot P)}}{1-\kappa^{2}}.
\end{equation}

To find the refracting part of the oval, 
from the physical constraint for refraction (Lemma 3.1), we must have
$x\cdot \left( \dfrac{P-h(x,P,b)x}{|P-h(x,P,b)x|}\right)\geq \kappa$,
and by (4.1) it is reduced to
\begin{equation}\label{eq:xdotPbiggerthanb}
x\cdot P\geq b.
\end{equation}



\begin{figure}[htp]
\begin{center}
    \subfigure[$|X|+2/3 |X-P|=1.4-1.9$, $P=(2,0)$]{\label{fig:edge-a}\includegraphics[width=2.9in]{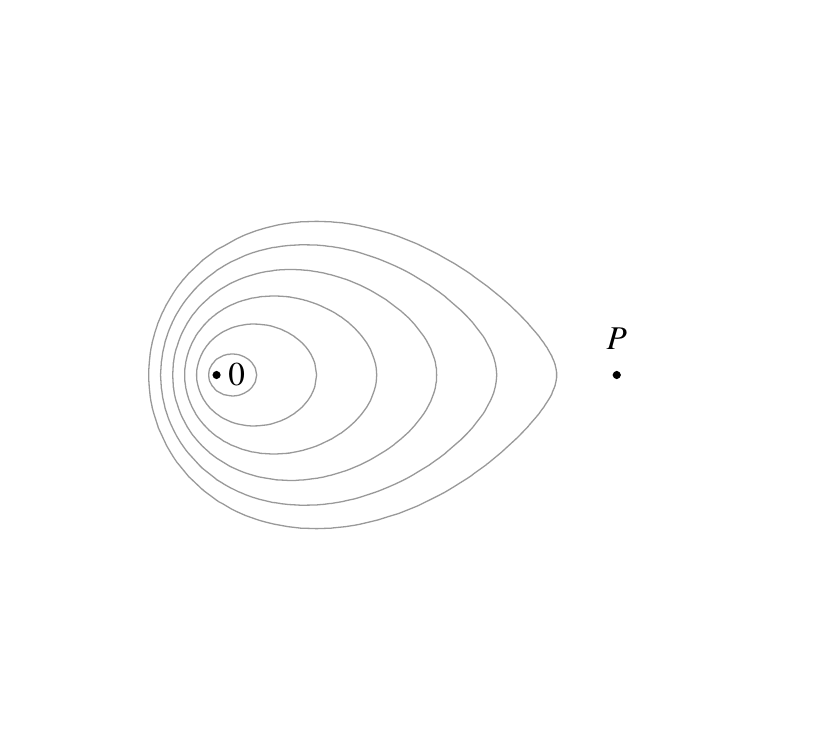}}
    \subfigure[$|X|+2/3 |X-P|=1.7$, $P=(2,0)$]{\label{fig:edge-b}\includegraphics[width=2.9in]{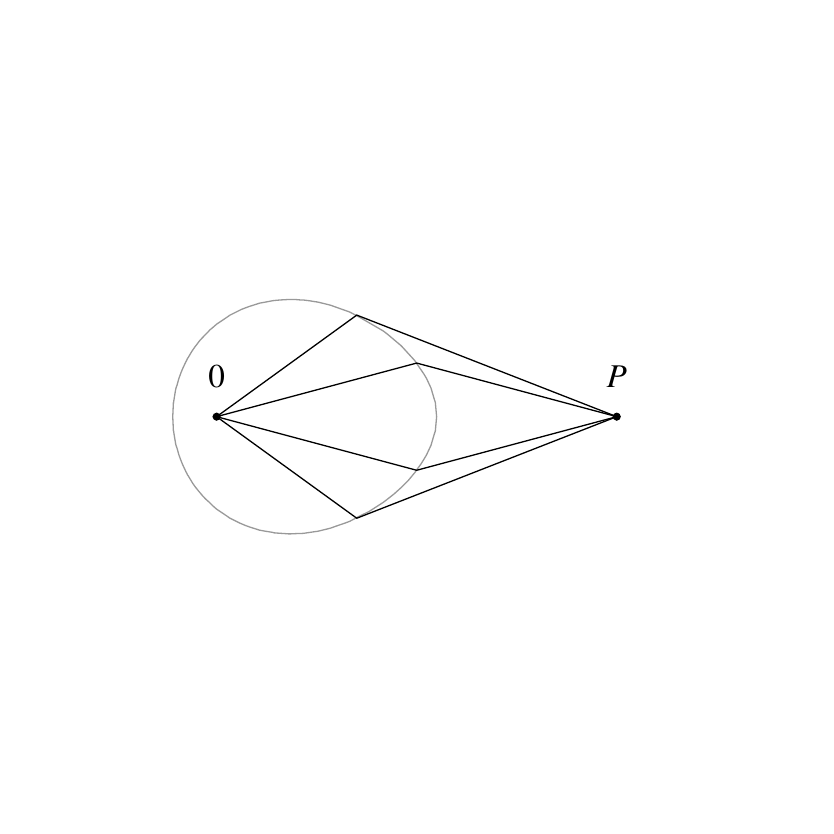}} 
\end{center}
  \caption{Cartesian ovals $\kappa<1$, e.g., glass to air}
  \label{fig:ovalskappa<1}
\end{figure}

The estimates of $h(x,P,b)$ are contained in the following lemma.

\begin{lemma}\label{lm:minandmaxofovalskappa<1}
Let $0<\kappa <1$, $h(x,P,b)$ given by \eqref{eq:polareqovalk<1}, and assume that $\kappa |P|<b<|P|$.
Then we have
\begin{equation}\label{eq:lowerestimateofrho}
\min_{x\in S^{n-1}}h(x,P,b)
=\dfrac{b-\kappa |P|}{1+\kappa},
\qquad 
\text{and}\qquad \max_{x\in S^{n-1}}h(x,P,b) =\dfrac{b-\kappa |P|}{1-\kappa}.
\end{equation}
We also have
\begin{equation}\label{eq:upperestimateofrho}
\min_{x\in S^{n-1}}|P-h(x,P,b)x|
=\dfrac{|P|-b}{1-\kappa}
=
\min_{x\cdot P\geq b}|P-h(x,P,b)x|
=\left|P-h\left(\frac{P}{|P|},P,b\right)\frac{P}{|P|}\right|,
\end{equation}
and
\begin{equation}\label{eq:maxofp-hforx.P>b}
\max_{x\cdot P\geq b}|P-h(x,P,b)x|=
\dfrac{\sqrt{|P|^{2}-b^{2}}}{\sqrt{1-\kappa^{2}}}.
\end{equation}
\end{lemma}
\begin{proof}
We write 
\begin{equation}\label{eq:writinghdifferent}
h(x,P,b)= \dfrac{b^{2}-\kappa^{2} |P|^{2}}{(b-\kappa^{2} x\cdot P)+\sqrt{\Delta(x\cdot P)}}
\end{equation}
and let $g (t)=(b-\kappa^{2} t)+\sqrt{\Delta(t)}$. We have $g$ is decreasing for $-|P|\leq t \leq |P|$, and 
so $g(-|P|)\geq g(x\cdot P)\geq g(|P|)$.
Hence $\dfrac{b^{2}-\kappa^{2} |P|^{2}}{g(-|P|)}\leq h(x,P,b) \leq \dfrac{b^{2}-\kappa^{2} |P|^{2}}{g(|P|)}$
and calculating $g(-|P|)$ and $g(|P|)$ the estimates in \eqref{eq:lowerestimateofrho} follow.

To prove \eqref{eq:upperestimateofrho},
since $\kappa |P-h(x,P,b)x|=b-h(x,P,b)$, the first equality follows from the right identity in \eqref{eq:lowerestimateofrho}.
To show the second identity in \eqref{eq:upperestimateofrho}, notice that since the oval is convex and symmetric with respect to the line joining 0 and $P$ we have that
$\min_{x\in S^{n-1}}|P-h(x,P,b)x|$ is attained at $x=P/|P|$.
In particular, this gives the explicit value of the distance from $P$ to the oval.

To prove \eqref{eq:maxofp-hforx.P>b} we have $\max_{x\cdot P\geq b}|P-h(x,P,b)x|=\dfrac{1}{\kappa}\left(b- \min_{x\cdot P\geq b} h(x,P,b) \right)$, and we claim that $\min_{x\cdot P\geq b} h(x,P,b)=h(z,P,b),$ for all $z\cdot P=b$.
In fact, this follows from \eqref{eq:writinghdifferent} since $g(x\cdot P)\leq g(b)$, obtaining
$h(x,P,b)\geq b- \dfrac{
\sqrt{ \Delta(b)}}{1-\kappa^{2}}=b-\dfrac{\kappa}{\sqrt{1-\kappa^2}}\sqrt{|P|^2-b^2}$. 

\end{proof}

\begin{remark}\rm
If $|P|\to \infty$, then the oval converges to an ellipsoid which is the surface having the uniform refraction property in the far field case, see \cite{gutierrez-huang:farfieldrefractor}.
In fact, if $m=P/|P|$ and $b=\kappa |P|+C$ with $C$ positive constant we have
\begin{align*}
h(x,P,b)&=\dfrac{b^{2}- \kappa^{2}|P|^{2}}{b-\kappa^{2}x\cdot P+ \sqrt{\Delta(x\cdot P)}}\\
&=
\dfrac{C(2\kappa |P| +C)}{(\kappa |P|-\kappa^{2}x\cdot m |P|+C)
+\sqrt{(\kappa |P|-\kappa^{2}x\cdot m |P|+C)^{2}-(1-\kappa^{2})C(2\kappa |P|+C)}}\\
&\quad \to 
\dfrac{2\kappa C}{(\kappa -\kappa^{2}x\cdot m)+\sqrt{(\kappa -\kappa^{2}x\cdot m)^{2}}}
=
\dfrac{C}{1-\kappa x\cdot m}
\end{align*}
as $|P|\to \infty$.
\end{remark}

\subsection{Case $\kappa >1$}\label{subsect:estimatesofovalskappa>1}
In this case we must have $|P|\leq b$, and in case $b=|P|$ the oval reduces to the point $P$. Also $b< \kappa |P|$, since otherwise the points $0,P$ are inside the oval or
$0$ is on the oval, and therefore there cannot be refraction if $b\geq \kappa |P|$.
So to have refraction we must have $|P|<b<\kappa |P|$ and so the point $P$ is inside the oval and $0$ is outside the oval.

Rewriting $\rho$ in \eqref{eq:definitionofrho} we get that
\[
\rho_{\pm}(x)=\dfrac{(\kappa^{2} x\cdot P-b)\pm 
\sqrt{(\kappa^{2}x\cdot P-b)^{2}- (\kappa^{2}-1)(\kappa^{2}|P|^{2}-b^{2})}}{\kappa^{2}-1}.
\]
Now $\Delta(x\cdot P)\geq 0$ amounts
$x\cdot P\geq \dfrac{b+\sqrt{(\kappa^{2}-1)(\kappa^{2}|P|^{2}-b^{2})}}{\kappa^{2}}$,
by Noticing that $\rho_{\pm}(x)<0$ if $\kappa^{2}x\cdot P-b<0$.
We have that $\rho_{-}(x)\leq \rho_{+}(x)\leq
\dfrac{(\kappa^{2}|P|-b)+\sqrt{\Delta(|P|)}}{\kappa^{2}-1}=\dfrac{\kappa |P|+b}{\kappa +1}<b$.
To have refraction, by the physical constraint we need to have
$x\cdot \dfrac{P-x\rho_{\pm}(x)}{|P-x\rho_{\pm}(x)|}\geq 1/\kappa$,
which is equivalent to $\kappa^{2}x\cdot P-b\geq (\kappa^{2}-1)\rho_{\pm}(x)$.
Therefore, the physical constraint is satisfied only by $\rho_{-}$.

For $|P|<b<\kappa |P|$, the refracting piece of the oval is then
given by
\begin{equation}\label{eq:equationovalk>1}
\mathcal O(P,b)
=
\left\{h(x,P,b)x:x\cdot P\geq  \dfrac{b+\sqrt{(\kappa^{2}-1)(\kappa^{2}|P|^{2}-b^{2})}}{\kappa^{2}}\right\}
\end{equation}
with 
\begin{equation}\label{eq:polareqovalk>1}
h(x,P,b)=\rho_{-}(x)=
\dfrac{(\kappa^{2} x\cdot P-b)- 
\sqrt{(\kappa^{2}x\cdot P-b)^{2}- (\kappa^{2}-1)(\kappa^{2}|P|^{2}-b^{2})}}{\kappa^{2}-1}.
\end{equation}
Let us define
\begin{equation}\label{eq:definitionofIPb}
I(P,b):=\dfrac{b+\sqrt{(\kappa^{2}-1)(\kappa^{2}|P|^{2}-b^{2})}}{\kappa^{2}|P|},
\end{equation}
and let 
\begin{equation}\label{eq:setofdirectionskappa>1}
\Gamma(P,b)=\{x\in S^{n-1}:x\cdot P\geq I(P,b)|P|\},
\end{equation}
that is, $\Gamma(P,b)$ denotes the set of directions in $\mathcal O(P,b)$.

We notice that $I(P,b)$
is decreasing as a function of $b$ and tends to one when $b\to |P|^{+}$, and 
tends to $1/\kappa$ when $b\to (\kappa |P|)^{-}$.

If $|P|\to \infty$, then the oval $\mathcal O(P,b)$ converges to the semi hyperboloid appearing in the far field refraction problem when $\kappa>1$, see \cite{gutierrez-huang:farfieldrefractor}.
Indeed, let $m=\dfrac{P}{|P|}\in S^{n-1}$ and $b=\kappa |P|-a$ with $a>0$ a constant.
Then we have
\[
\dfrac{b+\sqrt{(\kappa^{2}-1)(\kappa^{2}|P|^{2}-b^{2})}}{\kappa^{2}|P|}
=\dfrac{\kappa |P|-a+\sqrt{(\kappa^{2}-1)(\kappa^{2}|P|^{2}-(\kappa |P|-a)^{2})}}{\kappa^{2}|P|}\to \dfrac{1}{\kappa}
\]
as $|P|\to \infty$.
On the other hand, if $x\cdot m> 1/\kappa$, we get
\begin{align*}
h(x,P,b)&=\dfrac{a(2\kappa |P|-a)}{(\kappa^{2}|P|x\cdot m-\kappa |P|+a)
+
\sqrt{(\kappa^{2}|P|x\cdot m-\kappa |P|+a)^{2}-(\kappa^{2}-1)a(2\kappa |P|-a)}}\\
&\qquad \to
\dfrac{a2\kappa}{\kappa^{2}x\cdot m-\kappa+\sqrt{(\kappa^{2}x\cdot m -\kappa)^{2}}}=
\dfrac{a}{\kappa x\cdot m-1},
\end{align*}
as $|P|\to \infty$.

\begin{figure}[htp]
\begin{center}
    \subfigure[$|X|+3/2 |X-P|=2.9-2.4$, $P=(2,0)$]{\label{fig:edge-a}\includegraphics[width=2.9in]{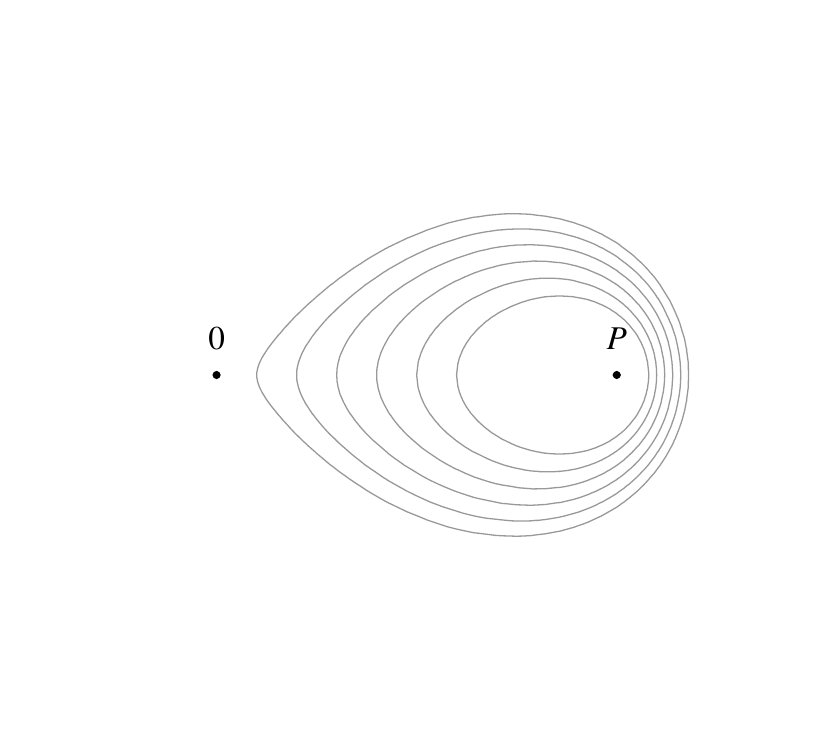}}
    \subfigure[$|X|+3/2 |X-P|=2.7$, $P=(2,0)$]{\label{fig:edge-b}\includegraphics[width=2.9in]{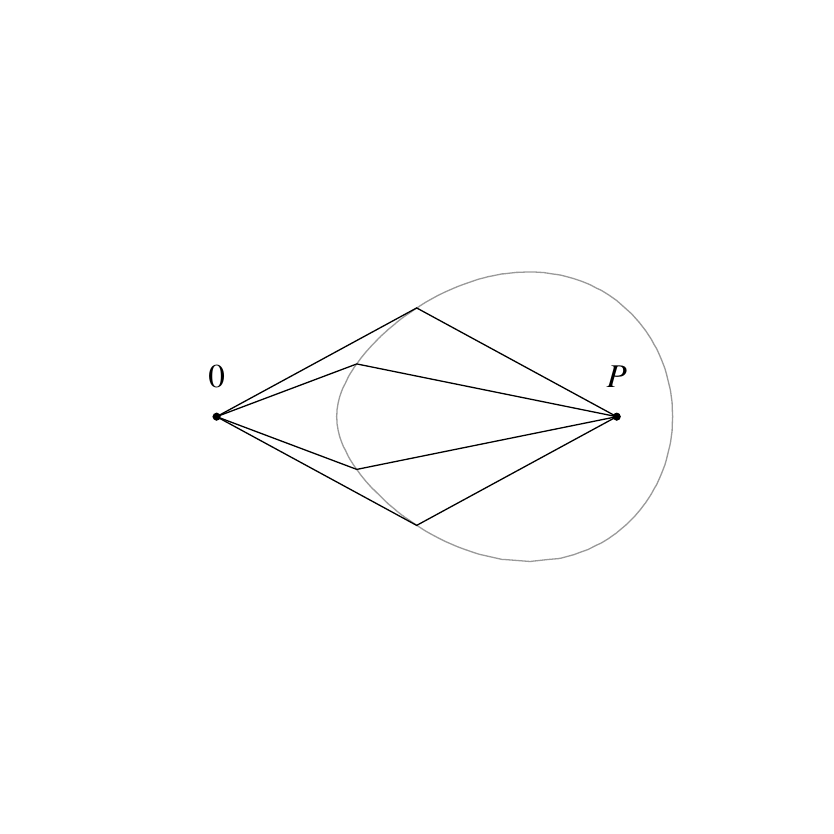}} 
\end{center}
  \caption{Cartesian ovals $\kappa>1$, e.g., air to glass}
  \label{fig:ovalskappa>1}
\end{figure}

The following lemma gives estimates for the size of $h(x,P,b)$.
\begin{lemma}\label{lm:estimateofovalsk>1}
Let $\kappa>1$, $h(x,P,b)$ is given by \eqref{eq:polareqovalk>1}, assume $|P|<b<\kappa |P|$,
and $\Gamma(P,b)$ given by \eqref{eq:setofdirectionskappa>1}.
We have 
\begin{enumerate}
\item[(a)] $\displaystyle \min_{x\in \Gamma(P,b)}h(x,P,b)=
\dfrac{\kappa |P|-b}{\kappa -1}$;
\item[(b)] $\displaystyle \max_{x\in \Gamma(P,b)}h(x,P,b)=
\dfrac{\sqrt{\kappa^{2}|P|^{2}-b^{2}}}{\sqrt{\kappa^{2} -1}}\leq \sqrt{2|P|}\sqrt{\dfrac{\kappa |P|-b}{\kappa -1}}$;
\item[(c)] $\displaystyle \dfrac{b-|P|}{\kappa} \leq |P-h(x,P,b)x|\leq
\dfrac{b-|P|}{\kappa-1}$, \quad for $x\in \Gamma(P,b)$;
\item[(d)] 
The following inequalities hold: 
\begin{equation*}
\dfrac{\sqrt{\kappa |P|-b}}{\sqrt{|P|}}
\dfrac{\kappa^{2}+\kappa -2}{2\kappa \sqrt{2\kappa (\kappa^{2}-1)}}
\leq I(P,b)-\dfrac{1}{\kappa}
\leq
\dfrac{\sqrt{\kappa |P|-b}}{\sqrt{|P|}}
\dfrac{2\sqrt{\kappa -1}}{\kappa},
\end{equation*}
where $I(P,b)$ is given by \eqref{eq:definitionofIPb}.
\end{enumerate}
\end{lemma}
\begin{proof}
(a) The minimum of $h(x,P,b)$ is attained when $x=P/|P|$. So
\[
\displaystyle \min_{x\in \Gamma(P,b)}h(x,P,b)
=h(P/|P|,P,b)=
\dfrac{\kappa^{2}|P|^{2}-b^{2}}{(\kappa^{2}|P|-b)+\sqrt{\Delta(|P|)}}
=\dfrac{\kappa |P|-b}{\kappa -1}.
\]

(b) The maximum of $h(x,P,b)$ on $\Gamma(P,b)$ is attained when 
\[
x\cdot P= \dfrac{b+\sqrt{(\kappa^{2}-1)(\kappa^{2}|P|^{2}-b^{2})}}{\kappa^{2}}, 
\]
that is, when $\Delta(x\cdot P)=0$ and the formula follows.

(c) We have
\begin{align*}
\min_{x\in \Gamma(P,b)}|P-h(x,P,b)x|&= \min_{x\in \Gamma(P,b)}
\left( \dfrac{b-h(x,P,b)}{\kappa}\right)
=
\dfrac{b-\max_{x\in \Gamma(P,b)}h(x,P,b)}{\kappa}\\
&=
\dfrac{b\sqrt{\kappa^{2}-1}-\sqrt{\kappa^{2}|P|^{2}-b^{2}}}{\kappa \sqrt{\kappa^{2}-1}}
\geq 
\dfrac{b\sqrt{\kappa^{2}-1}- |P|\sqrt{\kappa^{2}-1}}{\kappa \sqrt{\kappa^{2}-1}}.
\end{align*}
Furthermore, set $\rho(x)=h(x, P, b)$ and then $|P|+(\kappa-1)|P-\rho(x)x|\leq \rho(x)+\kappa |P-\rho(x)x|=b$.
It yields $|P-\rho(x)x|\leq \dfrac{b-|P|}{\kappa-1}$.

(d) It follows writing
\[
I(P,b)-\dfrac{1}{\kappa}
=
\dfrac{\sqrt{\kappa |P|-b}}{\kappa^{2}|P|}\cdot 
\dfrac{  (\kappa^{2}-2)\kappa |P| + \kappa^{2}b}{
\sqrt{(\kappa^{2}-1)(\kappa |P| +b)} + \sqrt{\kappa |P| -b}}
\]
and noticing that $(\kappa^{2}+\kappa -2)\kappa |P|
<(\kappa^{2}-2)\kappa |P| +\kappa^{2}b<2(\kappa^{2}-1)\kappa |P|$, since
$|P|<b<\kappa |P|$.
\end{proof}

\begin{remark}\rm
If $b\to (\kappa |P|)^{-}$, then, by Lemma \ref{lm:estimateofovalsk>1}(a),
$\mathcal O(P,b)$ approaches zero.
\end{remark}

\begin{remark}\rm
If $b\to |P|^{+}$, then $\mathcal O(P,b)$ shrinks to $P$. Because
$|P|+(\kappa -1)|P-X|\leq |X|+\kappa |P-X|=b$ for $X\in \mathcal O(P,b)$ and so
$|P-X|\leq \dfrac{b-|P|}{\kappa -1}$.
\end{remark}

\setcounter{equation}{0}
\section{Near field refractor problem for $\kappa<1$}\label{sec:definitionofweaksolution}

\subsection{Formulation of problem}
Let $\Omega\subset S^{n-1}$ be a domain with $|\partial \Omega|=0$
(measure in the sphere). Let $D\subset \R^{n}$ be a ``target'' domain that we want to illuminate and suppose it is contained in an $n-1$ dimensional hypersurface, and assume $\overline{D}$ is compact, and $0\notin \overline{D}$. 
Points in the sphere will be denoted with lower case letters and points in $\R^{n}$ by capitals.

We make the following assumptions on $\Omega$ and $D$:
\begin{enumerate}
\item \label{assumption1} 
There exists $\tau$ with $0<\tau < 1-\kappa$ such that $x\cdot P\geq (\kappa +\tau)|P|$ for all $x\in \overline{\Omega}$ and all $P\in \overline{D}$.
\item \label{assumption2} Let $0< r_{0}\leq \dfrac{\tau}{1+\kappa}\dist(0,\overline{D})$ and consider the cone in $\R^{n}$ 
\[Q_{r_{0}}=\{tx: x\in \overline{\Omega}, 0< t \leq r_{0}\}.\]
For each $m\in S^{n-1}$ and for each $X\in Q_{r_{0}}$ 
we assume that $\overline{D}\cap  \{X+tm: t\geq 0\}$ contains at most one point.
That is, for each $X\in Q_{r_{0}}$ each ray emanating from $X$ intersects $\overline{D}$ at most in one point.
\end{enumerate}

Given $P\in \R^{n}$ and $\kappa |P|<b<|P|$, keeping in mind \eqref{eq:polareqovalk<1} and \eqref{eq:xdotPbiggerthanb}, a refracting oval is the set 
\[
\mathcal O(P,b)=\{h(x,P,b)\,x: x\in S^{n-1}, x\cdot P\geq b\}
\]
where 
\[
h(x,P,b)=\dfrac{(b-\kappa^{2} x\cdot P)- 
\sqrt{(b- \kappa^{2}x\cdot P)^{2}- (1-\kappa^{2})(b^{2}-\kappa^{2}|P|^{2})}}{1-\kappa^{2}}.
\]
\begin{definition}\label{def:definitionofrefractorkappa<1}
Let $\mathcal S=\{x\rho(x): x\in \overline{\Omega}\}\subset Q_{r_{0}}$ be a surface.
We say that $\mathcal S$ is a near field refractor if for any point $y\rho(y)\in \mathcal S$ there exist $P\in \overline{D}$ and $b>0$ such that the refracting oval 
$\mathcal O(P,b)$ supports $\mathcal S$ at $y\rho(y)$, i.e. $\rho(x)\leq h(x,P,b)$ for all $x\in \overline{\Omega}$ with equality at $x=y$.

The near field refractor mapping associated with  $\mathcal S$ is defined by
\begin{equation}\label{eq:definitionofrefractormapping}
\mathcal R_{\mathcal S} (x)
=
\{P\in \overline{D}:  \text{there exists a supporting oval $\mathcal O(P,b)$ to $\mathcal S$ at $\rho(x)x$}
\}.
\end{equation}
\end{definition}

The definition implies that if $\mathcal O(P,b)$ is a supporting oval, then the openning of $\mathcal O(P,b)$ is wider than $\overline\Omega$,
i.e., $x\cdot P\geq b$ for all $x\in \overline{\Omega}$.

\begin{remark}\label{rmk:refractorissurjective}\rm
If $\mathcal S$ is a near field refractor, then $\mathcal R_{\mathcal S}(\overline{\Omega})=\overline{D}$.
Indeed, let $P\in \overline{D}$ and $b_{0}=(\kappa +\tau) |P|$.
Then from the left identity in \eqref{eq:lowerestimateofrho} and the assumption on $r_{0}$ in 
\ref{assumption2} we get that $r_{0}\leq \dfrac{b_{0}- \kappa |P|}{1+\kappa}\leq h(x,P,b_{0})$ for $x\in \overline{\Omega}$. Also from \ref{assumption1}, we have
$x\cdot P\geq b_{0}$. 
Hence $\mathcal S\subset Q_{r_{0}}$ enclosed by $\mathcal O(P,b_{0})$.
Let
\[
b_{1}=\inf \{b: \rho(x)\leq h(x,P,b),\, x\cdot P\geq b \quad \forall x\in \overline{\Omega}\}.
\]
Thus, the oval $\mathcal O(P,b_{1})$ supports $\mathcal S$ at some $y\in \overline{\Omega}$.

We point out that if $\mathcal S=\mathcal O(P,b)$, then $\overline D\setminus \{P\}\subset \mathcal R_{\mathcal S}(\partial\Omega)$.
\end{remark}

\begin{lemma}\label{rmk:uniformlipschitz}
If $\mathcal S$ is a near field refractor with defining function $\rho(x)$,
then $\rho$ is Lipschitz continuous in $\overline{\Omega}$ with a Lipschitz constant depending only on $\kappa$ ad $\tau$ in 
the assumptions \ref{assumption1} and \ref{assumption2} and $\max_{\overline D}|P|$.
\end{lemma}
\begin{proof}
Indeed, given $x_{0}\in \overline{\Omega}$, $\mathcal S$ has a supporting oval
$h(x,P,b)$ at $\rho(x_{0})x_{0}$ with $P\in \overline{D}$. Then
\begin{align*}
\rho(x)-\rho(x_{0})&\leq h(x,P,b)-h(x_{0},P,b)\\
&=
\dfrac{1}{1-\kappa^{2}}\left( (b-\kappa^{2}x\cdot P)-(b-\kappa^{2}x_{0}\cdot P)
+\sqrt{\Delta(x_{0}\cdot P)}-\sqrt{\Delta(x\cdot P)}\right)\\
&=
\dfrac{1}{1-\kappa^{2}}\left( I + II \right),
\end{align*}
where $\Delta$ is given by \eqref{eq:definitionofDelta}. 
We have $|I|\leq C(\kappa)|P|\cdot |x-x_{0}|$ and 
\begin{align*}
II
&=\dfrac{\Delta(x_{0}\cdot P) -\Delta (x\cdot P)}{\sqrt{\Delta(x_{0}\cdot P)}+\sqrt{\Delta(x \cdot P)}}.
\end{align*}
We estimate $\Delta(x_{0}\cdot P)$ from below.  Obviously,
the function $\Delta(t)$ has a minimum at $t=b/\kappa^{2}$, is increasing in $(b/\kappa^{2},+\infty)$ and decreasing in $(-\infty, b/\kappa^{2})$.
Since $\kappa<1$, we have $b/\kappa <b/\kappa^{2}$, and so
$\Delta$ decreases in the interval $[b,b/\kappa]$.
Since $|P|\in (b,b/\kappa)$ and $b\leq x\cdot P\leq |P|$, 
\[
\min_{[b,|P|]}\Delta (t)=\Delta(|P|)=\kappa^{2}\,(|P|-b)^{2},
\]
and therefore
\begin{equation}\label{eq:lowerestimateDelta}
\Delta(x\cdot P)\geq \kappa^{2}\,(|P|-b)^{2}, \qquad \text{for $x\cdot P\geq b$}.
\end{equation}
From \eqref{eq:maxofp-hforx.P>b}
\[
|P|-b\geq \dfrac{1-\kappa^{2}}{2|P|}\, |P-h(x,P,b)x|^{2},
\quad
\text{for all $x\cdot P\geq b$},
\]
which combined with \eqref{eq:lowerestimateDelta}
yields 
\begin{equation}\label{eq:estimatefrombelowofsqrtDelta}
\sqrt{\Delta (x\cdot P)}
\geq
\kappa \dfrac{1-\kappa^{2}}{2|P|}\, |P-h(x,P,b)x|^{2},
\text{for all $x\cdot P\geq b$}.
\end{equation}
Since $\rho(x_{0})=h(x_{0},P,b)$, and $\mathcal S\subset Q_{r_{0}}$, we get that
$h(x_{0},P,b)\leq r_{0}$.
From \ref{assumption2}, $|P|\geq r_{0}\dfrac{1+\kappa}{\tau}$. 
Therefore we obtain the estimate
\begin{align}\label{eq:lowerestimateosqrtDeltaneededinA4c}
\sqrt{\Delta (x_{0}\cdot P)}
&\geq
\kappa \dfrac{1-\kappa^{2}}{2|P|}\, |P-h(x_{0},P,b)x_{0}|^{2}
\geq
\kappa \dfrac{1-\kappa^{2}}{2|P|}\, \left( |P|-r_0\right)^{2}\\
&\geq
\dfrac{\kappa (1-\kappa)(1+\kappa -\tau)^{2}}{2(1+\kappa)}\,|P|.\notag
\end{align}
Clearly, 
$|\Delta(x_{0}\cdot P) -\Delta (x\cdot P)|\leq 
C(\kappa)|P|^2\,|x-x_{0}|$ which completes the proof of the lemma.
\end{proof}

\subsection{Application of the setup from Section \ref{sec:abstract setup} to the solution of the near field refractor problem} 
We apply the setup in that section with the spaces $X=\overline\Omega$, and $Y=\overline D$.
The Radon measure $\omega$ in $\overline\Omega$ there is now given by $\omega=fdx$ with 
$f\in L^1(\overline\Omega)$ nonnegative. 
If $\mathcal S$ is a near field refractor in the sense of Definition \ref{def:definitionofrefractorkappa<1}, then it is proved in 
Lemma \ref{measurabilityrefractormap} below that the map $\Phi=\mathcal R_{\mathcal S}\in C_s(\overline\Omega, \overline D)$.
From Lemma \ref{lm:radonmeasureontarget} we therefore obtain that
the set function
\begin{equation}\label{def:eqrefractormeasurekappa<1}
\mathcal M_{\mathcal S, f}(F):=\int_{\mathcal R_{\mathcal S}^{-1}(F)}f\,dx,
\end{equation}
is a Radon measure defined on $\overline D$. We call this measure 
{\it the near field refractor measure associated with $f$ and the refractor $\mathcal S$}.

Let $\mathcal S_\rho$ denote the near field refractor with defining radial function $\rho$ given by 
Definition \ref{def:definitionofrefractorkappa<1}.
We let $\mathcal F$ be the family of functions in $C^+(\overline\Omega)$ given by
$$\mathcal F=\{\rho(x): \mathcal S_\rho \text{ is a near field refractor}\}.$$ 
On $\mathcal F$ we define the mapping $\mathcal T$ by 
$$\mathcal T(\rho)=\mathcal R_{\mathcal S_\rho}.$$
To continue with the application of the results from Section \ref{sec:abstract setup}, 
we need to show also that $\mathcal T$ is continuous at each $\rho\in \mathcal F$ in the sense of Definition
\ref{def:definitionofmapTcontinuous}.
This is proved in Lemma \ref{lm:weakcompactness} below.

\begin{lemma}\label{measurabilityrefractormap}
For each near field refractor $\mathcal S$, we have $\mathcal R_{\mathcal S}\in C_s(\overline\Omega, \overline D)$.
\end{lemma}

\begin{proof}
By Remark \ref{rmk:refractorissurjective}, $\mathcal R_{\mathcal S}$ is surjective.  
Now show that $\mathcal R_{\mathcal S}(x)$ is single-valued for a.e. $x$ with respect to $\omega$.
If $\mathcal R_{\mathcal S}(x)$ contains more than one point, then 
$\mathcal S$ parameterized by $\rho$ has two distinct supporting
ovals  $\mathcal O(P_{1},b_{1})$ and $\mathcal O(P_{2},b_{2})$ at $\rho(x)x$ with $P_{1}\neq P_{2}$.
We claim that $\rho(x)x$ is a singular point of $\mathcal S$.
Otherwise, if $\mathcal S$ has tangent hyperplane ${\Pi}$ at $\rho(x)x$,
then $\Pi$ must coincide both with the tangent hyperplane of
$\mathcal O(P_{1},b_{1})$ and that of $\mathcal O(P_{2},b_{2})$ at $\rho(x)x$.
From the Snell law we get that 
$\dfrac{P_{1}-\rho(x)x}{|P_{1}-\rho(x)x|}=\dfrac{P_{2}-\rho(x)x}{|P_{2}-\rho(x)x|}:=m$, and so the ray through $X=\rho(x)x$ with direction $m$ contains $P_{1}, P_{2}$ and therefore $P_{1}=P_{2}$ from assumption \ref{assumption2}, a contradiction.
Since the graph of $\mathcal S$ is Lipschitz and $|\partial \Omega|=0$, the set of singular points of $\mathcal S$ has measure zero and therefore 
$\mathcal R_{\mathcal S}(x)$ is single-valued for a.e. $x\in \overline\Omega$.

To prove that $\mathcal R_{\mathcal S}$ is continuous, 
let $x_i\longrightarrow x_0$ and 
$P_i\in \mathcal R_{\mathcal S}(x_i)$.
Let $\mathcal O(P_{i},b_{i})$ be a supporting oval to $\mathcal S$
at $\rho(x_i)x_i$. Then
\begin{equation}\label{eq:inequalityofO_{i}}
\rho(x) +\kappa |P_{i}-\rho(x)x|\leq b_{i} \qquad \text{ for }x\in \overline\Omega,
\end{equation}
with equality at $x=x_i$ and $x\cdot P_{i}\geq b_{i}$ for all $x\in \overline{\Omega}$. Assume that $a_1\le \rho(x) \le r_0$
on $\overline\Omega$ for some constant $a_1>0$. 
From \eqref{eq:lowerestimateofrho} and \ref{assumption2} we get
$a_{1}(1-\kappa)+ \kappa |P_{i}|\leq b_{i}\leq \kappa |P_{i}| + r_0(1+\kappa)\leq (\kappa +\tau)|P_i|$.
Therefore selecting a subsequence we can assume that 
$P_i\longrightarrow P_0\in \overline D$ and $b_i\longrightarrow b_0$,
as $i \longrightarrow \infty$.
By taking limit in \eqref{eq:inequalityofO_{i}},
one obtains that the oval $\mathcal O(P_{0},b_{0})$ supports
$\mathcal S$ at $\rho(x_0)x_0$, $x\cdot P_{0}\geq b_{0}$, and
$P_0\in \mathcal R_{\mathcal S}(x_0)$.
\end{proof}

\begin{lemma}\label{lm:weakcompactness}
The refractor mapping $\mathcal T(\rho)=\mathcal R_{\mathcal S_\rho}$ is continuous at each $\rho\in \mathcal F$.
\end{lemma}
\begin{proof}
Suppose $\rho_j\longrightarrow \rho$ uniformly as $j\to \infty$. Let $x_0\in\overline\Omega$ and $P_j\in \mathcal R_{\mathcal S_{\rho_j}}(x_0)$. 
Then there exists $b_j$ such that 
$\rho_{j}(x)\leq h(x, P_{j},b_{j})$ for all $x\in \overline{\Omega}$ with equality 
at $x=x_{0}$ and with $x\cdot P_{j}\geq b_{j}$. As in the proof of Lemma  \ref{measurabilityrefractormap}, $\kappa |P_{j}|+ a (1-\kappa)\leq b_{j} \leq (\kappa + \tau) |P_{j}|$for some $a>0$, so
there exists a subsequence $P_{j_k}\longrightarrow P_0$ and $P_0\in \mathcal R_{\mathcal S_\rho}(x_0)$. 
\end{proof}

We therefore can apply Lemma \ref{lm:continuityofmeasuresabstractapproach} to obtain that the 
definition of refractor measure given in \eqref{def:eqrefractormeasurekappa<1} is stable by uniform limits,
i.e., if $\rho_j\to \rho$ uniformly, then $\mathcal M_{\mathcal S_{\rho_j}, f}\to \mathcal M_{\mathcal S_{\rho}, f}$ weakly.

To be able to apply Theorem \ref{thm:abstractcasediscritecase}, we next need to verify that the family $\mathcal F$ and the map $\mathcal T$ satisfy conditions (A1)-(A3) from Subsection 
\ref{subsect:generalconcavecase}.
Indeed, (A1) follows immediately from the definition of refractor.
Condition (A2) immediately follows from the definition of refractor.

It remains to verify (A3).
For that we use the estimates for ovals proved in Section \ref{sec:ovals}.
Indeed, with the notation in condition (A3) we will take 
\[
h_{t,y_0}(x)=h(x,P,b)
\]
with the understanding that $t=b$, and $y_0=P$, and 
$h(x,P,b)$ is the oval defined by \eqref{eq:polareqovalk<1}.
In other words, we will show that the family 
$$\left\{h(\cdot, P,b): \kappa |P|<b<\kappa |P| +(1-\kappa)r_0 \right\} \subset \mathcal F,$$
and verifies (A3), with $r_0$ from \ref{assumption2}.
Indeed, to show the inclusion, if $b<\kappa |P| +(1-\kappa)r_0 $,
then from \ref{assumption2} and \ref{assumption1} we have
$b<\kappa |P|+\dfrac{1-\kappa}{1+\kappa}\tau |P|\leq (\kappa+\tau)|P|\leq x\cdot P$ for all $x\in \Omega$.
So the oval $h(x,P,b)x$ refracts in $\Omega$ and in particular $P\in \mathcal T(h(\cdot, P,b))(x)$ for all $x\in \bar \Omega$, that is, (A3)(a) holds. 
%
Condition (A3)(b) is trivial. Condition (A3)(c) follows from the second identity in \eqref{eq:lowerestimateofrho}.
To verify (A3)(d), we notice that since $\Delta(x\cdot P)$ has a lower bound given in \eqref{eq:lowerestimateosqrtDeltaneededinA4c},
we obtain that 
$|h(x,P,b')-h(x,P,b)|\leq C\,|b'-b|$, with $C$ depending only on the constants in 
\ref{assumption1} and \ref{assumption2}.

%

The notion of weak solution is introduced through conservation
of energy.

\begin{definition}
A near field refractor $\mathcal{S}$ is a weak solution of the near field refractor problem for the case $\kappa<1$ with emitting illumination intensity
$f(x)$ on $\overline\Omega$ and prescribed refracted illumination
intensity $\mu$ on $\overline{D}$ if for any Borel set
$F\subset \overline{D}$
\begin{equation}
\mathcal M_{\mathcal S, f}(F)
=\int_{\mathcal R_{\mathcal S}^{-1}(F)}f\,dx=\mu(F).
\end{equation}
\end{definition}

We are now ready to apply Theorem \ref{thm:abstractcasediscritecase} to solve the near field refractor problem when the measure $\mu$ is a linear combination of deltas.


\subsubsection{Existence for sum of Dirac measures}

\begin{theorem}\label{thm:existencesumofdiracmasses}
Suppose \ref{assumption1} and \ref{assumption2} hold.
Let $P_{1},\cdots , P_{N}$ be distinct points in $\overline{D}$,  
$g_{1},\cdots , g_{N}$ are positive numbers, and $f\in L^{1}(\Omega)$ with $f>0$ a.e. in $\Omega$ such that 
\begin{equation}\label{eq:conservationofenergykappa<1}
\int_{\overline{\Omega}} f(x)\,dx=\sum_{i=1}^{N} g_{i}.
\end{equation}
Then, for each $b_{1}$ with $\kappa |P_{1}|< b_{1} < \kappa |P_{1}| + r_{0}
\dfrac{(1-\kappa)^2}{1+\kappa}$, there exists a unique $(b_{2},\cdots , b_{N})$ such that the poly-oval 
$
\mathcal S= \{\rho(x)x: x\in \overline{\Omega}\}
$ with
\begin{equation}\label{eq:definitionofpolyoval}
\rho(x)=\min_{1\leq i \leq N} h(x,P_{i},b_{i})
\end{equation}
is a weak solution to the near field refractor problem.
Moreover, $\mathcal M_{\mathcal S, f}\left(\{P_{i}\}\right)=g_{i}$ for $1\leq i \leq N$.
\end{theorem}
\begin{proof}
To prove the theorem, we apply Theorem \ref{thm:abstractcasediscritecase}. So we only need to verify that there exists 
$\rho_0(x)=\min_{1\leq i \leq N} h(x,P_{i},b^0_{i})$ satisfying
$\mathcal M_{\mathcal S_{\rho_0},f}(P_i)\le g_i$ for $2\leq i\leq N$.

Rewrite $b_1=\kappa |P_{1}| + (r_{0}-\sigma)\dfrac{(1-\kappa)^2}{1+\kappa}$ for some
$\sigma>0$.
Let $b_1^0=b_1$, and $b_{i}^0=\kappa |P_i|+(r_0-\sigma)(1-\kappa)$ for $2\leq i \leq N$. Then 
$\rho_0(x)=h(x,P_{1},b_{1}^0)$, because $h(x,P_{1},b_{1}^0)\leq \dfrac{b_{1}^0-\kappa |P_{1}|}{1-\kappa}
\leq \dfrac{(r_0-\sigma)(1-\kappa)}{1+\kappa} = \dfrac{b_{i}^0-\kappa |P_{i}|}{1+\kappa}\leq 
h(x,P_{i},b_{i}^0)$ for $2\leq i \leq N$, from \ref{assumption2} and \eqref{eq:lowerestimateofrho}.
Hence $\mathcal M_{\mathcal S_{\rho_0}, f}(\{P_{i}\})=0$ for $i\neq 1$.
The uniqueness follows from Theorem \ref{lm:partialcomparison}.
\end{proof}

\subsubsection{Existence in the general case}
\begin{theorem}\label{thm:existencegeneralcasekappa<1}
Assume conditions \ref{assumption1} and \ref{assumption2}.
Let $\mu$ be a Radon measure on $\overline{D}$, $f\in L^{1}(\Omega)$ with 
$f>0$ a.e., and satisfying the energy conservation condition
\[
\int_{\Omega} f(x)\,dx=\mu(\overline{D}).
\]

Then given $X_{0}\in Q_{r_{0}}$ with $0<|X_{0}|< \left( \dfrac{1-\kappa}{1+\kappa}\right)^{3}r_{0}$, there exists a weak solution of the near field refractor problem passing through $X_{0}$.
\end{theorem}

\begin{proof}
We assume first that $\mu=\sum_{i=1}^{N} g_{i}\delta_{P_{i}}$, with $g_{i}>0$ and 
$P_{i}$ distinct points in $\overline D$.
From Theorem \ref{thm:existencesumofdiracmasses},  
given $b_{1}\in \left(\kappa |P_{1}|, \kappa |P_{1}| + r_0\dfrac{(1-\kappa)^2}{1+\kappa}\right)$ there exists a unique $(b_{2},\cdots , b_{N})$ such that 
$\mathcal S$, defined by the radial function $\rho(x,b_{1})=\min_{i}h(x,P_{i},b_{i})$, 
is a weak solution to the near field refractor problem. By the comparison Theorem \ref{lm:partialcomparison},  the function $\rho(x,b_{1})$ is increasing  in $b_1$ and continuous 
for $(x,b_{1})\in \overline{\Omega} \times \left(\kappa |P_{1}|, \kappa |P_{1}| + r_0\dfrac{(1-\kappa)^2}{1+\kappa}\right)$.   Let $t_1=\kappa |P_{1}| + (r_0-\sigma)\dfrac{(1-\kappa)^2}{1+\kappa}$ for
$0<\sigma<r_0$.   We shall first prove that 
\begin{equation}\label{eq:estimateofrho}
\rho(x,t_1)\geq \left( \dfrac{1-\kappa}{1+\kappa}\right)^{3}(r_{0}-\sigma),\quad
\forall x\in \overline{\Omega}.
\end{equation}  
We have $\rho(x,t_1)=\min_{1\leq i\leq N} h(x,P_{i},b_{i})$ with $b_{1}=t_1$ and some $b_{2},\cdots, b_{N}$.
From \eqref{eq:lowerestimateofrho} 
\begin{equation}\label{eq:lowerestimateofh}
\dfrac{b_{i}-\kappa |P_{i}|}{1-\kappa}
\geq h(x,P_{i},b_{i})\geq \rho(x,t_1),\quad i=1,\cdots ,N;\quad \forall x\in \overline{\Omega}.
\end{equation}
Also there exists $x_{1}\in \overline{\Omega}$ such that 
$\rho(x_{1},t_1)=h(x_{1},P_{1}, t_1)$ and then again by \eqref{eq:lowerestimateofrho},
$\rho(x_{1},t_1)\geq \left(\dfrac{1-\kappa}{1+\kappa}\right)^2(r_{0}-\sigma)$.
Hence from \eqref{eq:lowerestimateofh} we get
\begin{equation}\label{eq:estimateofr0}
\dfrac{b_{i}-\kappa |P_{i}|}{1-\kappa}
\geq \left(\dfrac{1-\kappa}{1+\kappa}\right)^2(r_{0}-\sigma),\quad i=1,\cdots ,N.
\end{equation}
Once again by  \eqref{eq:lowerestimateofrho},
$h(x,P_{i},b_{i})\geq \dfrac{b_{i}-\kappa |P_{i}|}{1+\kappa}$,
which combined with \eqref{eq:estimateofr0} yields
\[
h(x,P_{i},b_{i})\geq \left( \dfrac{1-\kappa}{1+\kappa}\right)^{3}(r_{0}-\sigma),
\quad i=1,\cdots ,N,
\]
and hence \eqref{eq:estimateofrho} follows.
On the other hand, $\rho(x,b_{1})\leq h(x,P_{1},b_{1})\leq \dfrac{b_{1}-\kappa |P_{1}|}{1-\kappa}$ by \eqref{eq:lowerestimateofrho} for all $(x,b_{1})\in 
\overline{\Omega} \times \left(\kappa |P_{1}|, \kappa |P_{1}| +  r_0\dfrac{(1-\kappa)^2}{1+\kappa}\right)$
 and hence given $\delta>0$, we get $\rho(x,b_{1})<\delta$ for all $x\in \overline{\Omega}$
as long as $b_{1}$ is sufficiently close to $\kappa |P_{1}|$.
Suppose now that $X_{0}\in Q_{r_{0}}$ with $0<|X_{0}|< \left( \dfrac{1-\kappa}{1+\kappa}\right)^{3}r_{0}$ and with $x_{0}=\dfrac{X_{0}}{|X_{0}|}\in \overline{\Omega}$.
Hence from \eqref{eq:estimateofrho} and the continuity
of $\rho(x_{0},\cdot) $, we obtain that there exists $b_{1}\in \left(\kappa |P_{1}|, \kappa |P_{1}| + 
 r_0\dfrac{(1-\kappa)^2}{1+\kappa}\right)$ such that $\rho(x_{0}, b_{1})=|X_{0}|$. 

For the general case of a Radon measure $\mu$ in $\overline D$, 
we choose a sequence of measures $\mu_{\ell}$ such that each one is a finite combination of Dirac measures and $\mu_{\ell}\to \mu$ weakly with
$\mu_{\ell}(\overline{D})=\mu(\overline{D})$.
From the above, let $\mathcal S_{\ell}$ be the near field refractor corresponding to the measure $\mu_{\ell}$ and parameterized by
$\rho_{\ell}(x)x$ and passing through the point $X_{0}$.
Thus, $x_{0}=\dfrac{X_{0}}{|X_{0}|}\in \overline{\Omega}$ and $\rho_{\ell}(x_{0})=|X_{0}|$. 
i.e., $|X_0|\in \text{Range }\rho_\ell$ for all $\ell$. 
We also notice that if $R_0\in \text{Range }\left(h(\cdot, P, b)\right)$, then by Lemma \ref{lm:minandmaxofovalskappa<1}
\[
\dfrac{1-\kappa}{1+\kappa}R_0\leq h(x, P, b) \leq \dfrac{1+\kappa}{1-\kappa}R_0.
\]
From Lemma \ref{rmk:uniformlipschitz} and the proof of Lemma 5.5, the family $\{\rho\in \mathcal F: C_0\leq \rho \leq C_1\}$ is compact.
Then applying Theorem \ref{thm:uniquenessabstractcase}
we obtain the existence of the desired solution.
\end{proof}

\setcounter{equation}{0}
\section{Near field refractor problem, existence of solutions for $\kappa>1$}\label{sec:existencekappa>1}

\subsection{Formulation of problem}
Let $\Omega\subset S^{n-1}$ be a domain with $|\partial \Omega|=0$
(measure in the sphere), and let $D\subset \R^{n}$ be a compact hypersurface with $0\not\in \overline{D}$.

We assume:
\begin{enumerate}
\item[(H3)]
$\inf_{x\in \overline{\Omega}, P\in \overline{D}}x\cdot \dfrac{P}{|P|}\geq \dfrac{1}{\kappa}+\tau$ for some $0<\tau <1-\dfrac{1}{\kappa}$. 
\item[(H4)]
Let $0< r_{0}< \dfrac{\kappa ^{2}\tau^{2}}{4(\kappa-1)^{2}}
\,\inf_{P\in D}|P|$ and consider the cone in $\R^{n}$ 
\[Q_{r_{0}}=\{tx: x\in \overline{\Omega}, 0<t \leq r_{0}\}.\]
For each $m\in S^{n-1}$ and for each $X\in Q_{r_{0}}$ 
we assume that $\overline{D}\cap  \{X+tm: t\geq 0\}$ contains at most one point.
That is, for each $X\in Q_{r_{0}}$ each ray emanating from $X$ intersects $\overline{D}$ at most in one point.
\end{enumerate}

Similarly with the case $\kappa<1$, but now keeping in mind \eqref{eq:equationovalk>1} and \eqref{eq:polareqovalk>1},
we define the notion of refractor when $\kappa>1$.
\begin{definition}\label{def:refractorkappa>1}
Let $\mathcal S=\{x\rho(x): x\in \overline{\Omega}\}\subset Q_{r_{0}}$ be a surface.
We say that $\mathcal S$ is a near field refractor if for any point $y\rho(y)\in \mathcal S$ there exist $P\in \overline{D}$ and $b>0$ such that the refracting oval $\mathcal O(P,b)$ supports $\mathcal S$ at $y\rho(y)$, i.e. $\rho(x)\geq h(x,P,b)$ for all $x\in \overline{\Omega}$ with equality at $x=y$, and
\begin{equation}\label{eq:inclusionomegainIPb}
\overline{\Omega}\subset \left\{x\in S^{n-1}: x\cdot \dfrac{P}{|P|}\geq I(P,b) \right\}
\end{equation}
where $I(P,b)$ is defined in \eqref{eq:definitionofIPb}.
The near field refractor map of $\mathcal S$ is defined by
\[
\mathcal R_{\mathcal S} (x)
=
\{P\in \overline{D}:  \text{there exists a supporting oval $\mathcal O(P,b)$ to $\mathcal S$ at $\rho(x)x$}
\}.
\]
\end{definition}

\begin{remark}\label{rmk:estimateofrhoandrhox0}\rm
If $\mathcal S=\{\rho(x)x: x\in \overline{\Omega}\}$ is a refractor, then for any $x,x_{0}\in \overline{\Omega}$ we have
\[
\rho(x)\leq \sqrt{2\sup_{P\in D}|P|}\, \sqrt{\rho(x_{0})}.
\]
Indeed, 
if $h(z,P_{x},b_{x})$ is a supporting oval at $\rho(x)x$, then from Lemma \ref{lm:estimateofovalsk>1} we have $\rho(x_{0})\geq h(x_{0},P_{x},b_{x})
\geq \dfrac{\kappa |P_{x}|-b_{x}}{\kappa -1}\geq \dfrac{h(x,P_{x},b_{x})^{2}}{2|P_{x}|}
=
\dfrac{\rho(x)^{2}}{2|P_{x}|}$.
\end{remark}

\begin{remark}\label{eq:estimateofIPb}\rm
$\mathcal R_{\mathcal S}(\overline{\Omega})=\overline{D}$ for any refractor
$\mathcal S=\{\rho(x)x:x\in \overline{\Omega}\}$.
To prove the remark, we first notice that if $P\in \overline{D}$ and $h(x_{0},P,b)\leq r_{0}$ for some $x_{0}\in \overline{\Omega}$, then $I(P,b)\leq \dfrac{1}{\kappa}+\tau$.
Indeed, we have from Lemma \ref{lm:estimateofovalsk>1}(d)
that $I(P,b)-\dfrac{1}{\kappa}
\leq \dfrac{2\sqrt{\kappa -1}}{\kappa}
\dfrac{\sqrt{\kappa |P|-b}}{\sqrt{|P|}}\leq \dfrac{2(\kappa -1)}{\kappa}\dfrac{\sqrt{r_{0}}}{\sqrt{|P|}}\leq \tau$ from the choice of $r_{0}$ in condition (H4).
Given $P\in \overline{D}$, let $\mathcal O(P,b_{1})$ be the oval with 
\[
b_{1}=\inf \{b\in (|P|,\kappa |P|): h(x,P,b)\leq \rho(x) \text{ in $\overline{\Omega}$}\}.
\]
Obviously, $h(x,P,b_{1})$ touches $\mathcal S$ at some $x_{1}\in \overline{\Omega}$.

\end{remark}
\begin{lemma}\label{rmk:uniformlipschitzkappa>1}
If $\mathcal S$ is a near field refractor with defining function $\rho(x)$,
then $\rho$ is Lipschitz in $\overline{\Omega}$ with a Lipschitz constant depending only on the constants in 
the assumptions (H3) and (H4).
\end{lemma}

\begin{proof}
As in the proof of Lemma \ref{rmk:uniformlipschitz}, the Lipschitz continuity and constant of $h(x,P,b)$ is reduced to 
the Lipschitz continuity and lower bound of $\Delta$, where
\begin{align*}
\Delta(t)&=(\kappa^2t-b)^2-(\kappa^{2}-1)(\kappa^2|P|^2-b^2),  \\
h(x,P,b)&=\dfrac{(\kappa^2x\cdot P-b)-\sqrt{\Delta(x\cdot P)}}
{\kappa^2-1}
\qquad\text{for }x\in \overline\Omega.
\end{align*}
By (H3) $x\cdot \dfrac{P}{|P|}\geq   \dfrac{1}{\kappa}+\tau$ for
$x\in \overline\Omega$.
From Lemma \ref{lm:estimateofovalsk>1}(a) and (d) we have
\begin{align*}
I(P,b)-\frac1\kappa
&=\dfrac{b+\sqrt{(\kappa^{2}-1)
(\kappa^{2}|P|^{2}-b^{2})}}{\kappa^{2}|P|}-\frac1\kappa  \\
&\le \frac{2\sqrt{\kappa-1}}\kappa \frac{\sqrt{\kappa|P|-b}}{\sqrt{|P|}} \\
&\le \frac{2(\kappa-1)}\kappa \frac{\sqrt{r_0}}{\sqrt{|P|}}\qquad
\text{if } h(x_0,P,b)\le r_0  \\
&\leq \tau-\sigma,
\end{align*}
where $\sigma=\tau-\dfrac{2(\kappa-1)}\kappa 
\dfrac{\sqrt{r_0}}{\sqrt{\inf_{P\in D}|P|}}>0$ by (H4).
\bigskip
Thus, if $h(x_{0},P,b)\leq r_{0}$, 
for $x\in \overline\Omega$ we have $x\cdot \dfrac{P}{|P|}\geq  I(P,b)+\sigma$  
and furthermore  
\begin{align}\label{eq:lowerestimateofdeltakappa>1}
\Delta(x\cdot P)
&\ge \left[\kappa^2|P|(I(P,b)+\sigma)-b\right]^2-
(\kappa^2-1)(\kappa^2|P|^2-b^2)\notag \\
&\ge \left[\kappa^2|P|I(P,b)-b\right]^2+(\kappa^2|P|\sigma)^2-
(\kappa^2-1)(\kappa^2|P|^2-b^2) \notag \\
&\ge \left(\kappa^2\inf_{P\in D}|P|\sigma\right)^2.
\end{align}
Therefore, the uniform Lipschitz continuity for ovals and refractors
contained in $Q_{r_0}$ on $\overline\Omega$ follows.
\end{proof}


\subsection{Application of the setup from Subsection \ref{subsect:convexcaseBIS} to the solution of the near field refractor problem with $\kappa>1$} 

We apply the setup in that subsection with the spaces $X=\overline\Omega$, and $Y=\overline D$.
The Radon measure $\omega$ in $\overline\Omega$ there is given by $\omega=fdx$ with 
$f\in L^1(\overline\Omega)$ nonnegative. 
If $\mathcal S$ is a near field refractor in the sense of Definition \ref{def:refractorkappa>1}, then it is proved in 
Lemma \ref{measurabilityrefractormapkappa>1} below that the map $\mathcal R_{\mathcal S}\in C_s(\overline\Omega, \overline D)$.
From Lemma \ref{lm:radonmeasureontarget} we therefore obtain that
the set function
\begin{equation}\label{def:eqrefractormeasure}
\mathcal M_{\mathcal S, f}(F):=\int_{\mathcal R_{\mathcal S}^{-1}(F)}f\,dx
\end{equation}
is a Radon measure on $\overline D$. We call this measure 
{\it the near field refractor measure associated with $f$ and the refractor $\mathcal S$}.

We next introduce the family $\mathcal F$. Let $\mathcal S_\rho$ denote the near field refractor with defining radial function $\rho$ given by 
Definition \ref{def:refractorkappa>1}.
We let $\mathcal F$ be the family of functions in $C^+(\overline\Omega)$ given by
$$\mathcal F=\{\rho(x): \mathcal S_\rho \text{ is a near field refractor}\}.$$ 
On $\mathcal F$ we introduce the mapping $\mathcal T$ by 
$$\mathcal T(\rho)=\mathcal R_{\mathcal S_\rho}.$$
To continue with the application of the results from Subsection \ref{subsect:convexcaseBIS}, 
we show in the next two lemmas that $\mathcal R_{\mathcal S}\in C_s(\overline\Omega, \overline D)$ and $\mathcal T$ is continuous at each $\rho\in \mathcal F$ in the sense of Definition
\ref{def:definitionofmapTcontinuous}.

\begin{lemma}\label{measurabilityrefractormapkappa>1}
For any near field refractor $\mathcal S$, we have $\mathcal R_{\mathcal S}\in C_{\mathcal S}(\overline\Omega, \overline D)$.
\end{lemma}

\begin{proof}
As in the proof of Lemma \ref{measurabilityrefractormap}, using (H4), one can show that $\mathcal R_{\mathcal S}(x)$ is single-valued for a.e. x with
respect to $\omega$.

To prove that $\mathcal R_{\mathcal S}$ is continuous, let $x_i\longrightarrow x_0$ and $P_i\in \mathcal R_{\mathcal S}(x_i)$.
Let $\mathcal O(P_{i},b_{i})$ be a supporting oval to $\mathcal S$ at $\rho(x_{i})x_{i}$.
We have $a\leq \rho(x_{i})= h(x_{i},P_{i},b_{i})\leq r_{0}$ and from Lemma 
\ref{lm:estimateofovalsk>1} (b) we get 
$a\leq \sqrt{2|P_{i}|} \sqrt{\dfrac{\kappa |P_{i}|-b_{i}}{\kappa -1}}$, and so
$b_{i}\leq \kappa |P_{i}|-\dfrac{a^{2}(\kappa -1)}{2\sup_{D}|P|}$.
On the other hand, by Lemma \ref{lm:estimateofovalsk>1} (c)
we have $|P_{i}-\rho(x_{i})x_{i}|\leq \dfrac{b_{i}-|P_{i}|}{\kappa -1}$.
Since $|P_{i}-\rho(x_{i})x_{i}|\geq |P_{i}|-r_{0}$, we obtain  
$b_{i}\geq  |P_i|+(\kappa -1)\left(\inf_{D}|P|-r_0\right)$.
Therefore selecting a subsequence we can assume that 
$P_i\longrightarrow P_0$ and $b_i\longrightarrow b_0$,
as $i \longrightarrow \infty$.
Taking limits
one obtains that the oval $\mathcal O(P_{0},b_{0})$ supports
$\mathcal S$ at $\rho(x_0)x_0$, $x\cdot \dfrac{P_{0}}{|P_{0}|}\geq I(P_{0},b_{0})$, and
$P_0\in \mathcal R_{\mathcal S}(x_0)$.  This completes the proof.
\end{proof}

By Lemma \ref{rmk:uniformlipschitzkappa>1} and modifying the proof of Lemma \ref{measurabilityrefractormapkappa>1}, 
we also obtain the following analogue of Lemma \ref{lm:weakcompactness}
when $\kappa>1$.

\begin{lemma}\label{lm:weakcompactnesskappa>1}
The refractor mapping $\mathcal T(\rho)=\mathcal R_{\mathcal S_\rho}$ is continuous at each $\rho\in \mathcal F$.
Moreover, for $0<C_0<C_1$, $\{\rho\in \mathcal F: C_0\leq \rho(x) \leq C_1\}$ is compact in $C(\overline \Omega)$.
\end{lemma}

To be able to apply Theorem \ref{thm:abstractcasediscritecaseconvexcaseBIS}, we next need to verify that the family $\mathcal F$ and the map $\mathcal T$ satisfy conditions (A1')-(A3') from Subsection 
\ref{subsect:convexcaseBIS}.
Indeed, (A1') and (A2') follow immediately from the Definition \ref{def:refractorkappa>1} of refractor, Lemma \ref{measurabilityrefractormapkappa>1}, and 
Lemma \ref{lm:weakcompactnesskappa>1}.

It remains to verify (A3').
For that we use the estimates for ovals proved in Subsection \ref{subsect:estimatesofovalskappa>1}.
Indeed, with the notation in condition (A3') we will take 
\[
h_{t,y_0}(x)=h(x,P,b)
\]
with the understanding that $t=b$, and $y_0=P$, and 
$h(x,P,b)$ is the oval defined by \eqref{eq:equationovalk>1} and \eqref{eq:polareqovalk>1}.
By Lemma \ref{lm:estimateofovalsk>1} and Remark \ref{eq:estimateofIPb}, we have that the family
$$\left\{h(\cdot, P,b): \kappa |P|-\dfrac{(\kappa-1)r_0^2}{2\sup_D|P|}<b<\kappa |P|  \right\} \subset \mathcal F,$$ with $r_0$ from (H4).
Obviously, $P\in \mathcal T(h(\cdot, P,b))(x)$ for all $x\in \bar \Omega$, that is, (A3')(a) holds. 
%
Condition (A3')(b) is trivial. Condition (A3')(c) follows Lemma \ref{lm:estimateofovalsk>1} parts (a) and (b).
To verify (A3')(d), we notice that from the lower bound \eqref{eq:lowerestimateofdeltakappa>1},
we obtain that 
$|h(x,P,b')-h(x,P,b)|\leq C\,|b'-b|$, with $C$ depending only on the constants in 
(H3) and (H4).

The notion of weak solution is again introduced through conservation
of energy.

\begin{definition}
A near field refractor $\mathcal{S}$ is a weak solution of the near field refractor problem for the case $\kappa>1$ with emitting illumination intensity
$f(x)$ on $\overline\Omega$ and prescribed refracted illumination
intensity $\mu$ on $\overline{D}$ if for any Borel set
$F\subset \overline{D}$
\begin{equation}
\mathcal M_{\mathcal S, f}(F)
=\int_{\mathcal {\mathcal R}_{\mathcal S}^{-1}(F)}f\,dx=\mu(F).
\end{equation}
\end{definition}

\subsubsection{Existence of solutions for sum of Dirac measures}
This follows from Theorem \ref{thm:abstractcasediscritecaseconvexcaseBIS}.
\begin{theorem}\label{thm:existencesumofdiracmasseskappa>1}
Suppose (H3) and (H4) hold.
Let $P_{1},\cdots , P_{N}$ be distinct points in $\overline{D}$,  
$g_{1},\cdots , g_{N}$ are positive numbers, and $f\in L^{1}(\Omega)$ with $f>0$ a.e. in $\Omega$ such that 
\begin{equation}\label{eq:conservationofenergy}
\int_{\overline{\Omega}} f(x)\,dx=\sum_{i=1}^{N} g_{i}.
\end{equation}
Then for each $b_{1}$ such that $\kappa |P_{1}|-\sigma<  b_{1} < \kappa |P_{1}|$,
with $\sigma=\dfrac{(\kappa -1)r_{0}^4}{8 ( \sup_{D}|P|)^3}$, there exist a unique $(b_{2},\cdots , b_{N})$ such that the poly-oval 
$
\mathcal S= \{\rho(x)x: x\in \overline{\Omega}\}
$ with
\begin{equation}\label{eq:definitionofpolyovalkappa>1}
\rho(x)=\max_{1\leq i \leq N} h(x,P_{i},b_{i})
\end{equation}
is a weak solution to the near field refractor problem.
Moreover, $\mathcal M_{\mathcal S, f}\left(\{P_{i}\}\right)=g_{i}$ for $1\leq i \leq N$.
\end{theorem}
\begin{proof}
By Theorem \ref{thm:abstractcasediscritecaseconvexcaseBIS}, it suffices to show that 
there exists $(b^0_1, \cdots, b^0_N)$ such that $h(x,P_1,b^0_1)\le \min_{2\le i\le N}h(x,P_i,b^0_i)$ 
for all $x\in \overline \Omega$.

Rewrite $b_1^0=b_1=\kappa|P_1|-\epsilon^2\sigma$ with $0<\epsilon<1$. Choose $b_i^0=\kappa|P_i|-\epsilon \dfrac{(\kappa -1)r_{0}^2}{2\sup_{D}|P|}$, $2\leq i\leq N$.  Then from Lemma \ref{lm:estimateofovalsk>1}(b), $h(x,P_1,b^0_1)\leq \sqrt{2|P_1|}\sqrt{\dfrac{\kappa |P_1|-b^0_1}{\kappa -1}}\leq \sqrt{2|P_1|}\sqrt{\dfrac{\epsilon^2\sigma}{\kappa -1}}$. 
On the other hand, from Lemma \ref{lm:estimateofovalsk>1}(a), we have for $2\leq i\leq N$ that
\[
h(x, P_i, b_i^0) \geq \dfrac{\kappa |P_i|-b_i^0}{\kappa-1}=\dfrac{\epsilon r_0^2}{2\sup_D|P|}.
\]
By the choice of $\sigma$, $h(x, P_1, b_1^0)\leq h(x, P_i, b_i^0)$ for $2\leq i\leq N$.

\end{proof}

\subsubsection{Existence in the general case}
\begin{theorem}\label{thm:generalcaseRadonmeasurekappa>1}
Assume conditions (H3) and (H4).
Let $\mu$ be a Radon measure on $\overline{D}$, $f\in L^{1}(\Omega)$ with 
$f>0$ a.e., and satisfying the energy conservation condition
\[
\int_{\Omega} f(x)\,dx=\mu(\overline{D}).
\]
Then given $X_{0}\in Q_{r_{0}}$ with $0<|X_{0}|< 
\dfrac{\sigma}{\kappa -1}
=
\dfrac{r_{0}^{4}}{8 \left( \sup_{D}|P|\right)^{3}}$, there exists a weak solution of the near field refractor problem passing through $X_{0}$.
\end{theorem}

\begin{proof}
We assume first that $\mu=\sum_{i=1}^{N} g_{i}\delta_{P_{i}}$, with $g_{i}>0$ and 
$P_{i}$ pairwise distinct points in $\overline D$.
From Theorem \ref{thm:existencesumofdiracmasseskappa>1}, given $b_{1}\in 
(\kappa |P_{1}|-\sigma, \kappa |P_{1}| )$ there exists a unique $(b_{2},\cdots , b_{N})$ such that 
$\mathcal S_{b_1}$, defined by the radial function $\rho(x,b_{1})
=\max_{i}h(x,P_{i},b_{i})$, is a weak solution to the near field refractor problem.
By the comparison Theorem \ref{lm:partialcomparisonconvex}, 
the function $\rho(x,b_{1})$ is decreasing in $b_1$ and continuous 
for $(x,b_{1})\in \overline{\Omega} \times (\kappa |P_{1}|-\sigma, \kappa |P_{1}| )$.
For small $\epsilon>0$, from Lemma \ref{lm:estimateofovalsk>1}(a) we have 
\begin{equation}\label{eq:estimateofrhokappa>1}
\rho(x,\kappa |P_{1}| - (\sigma-\epsilon))\geq 
h(x,P_{1}, \kappa |P_{1}|-(\sigma-\epsilon)) \geq
\dfrac{\sigma-\epsilon}{\kappa -1},\quad \forall x\in \overline{\Omega}.
\end{equation}  
On the other hand,
let $b_{i}(\epsilon)$, $2\leq i \leq N$ be the corresponding $b_{j}$'s to $b_{1}(\epsilon)=\kappa |P_{1}|-\epsilon$.
We have 
$\dfrac{\kappa |P_{i}|-b_{i}(\epsilon)}{\kappa -1}
\leq h(x,P_{i},b_{i}(\epsilon))\leq \rho(x,b_1(\epsilon))$.
Since $|\mathcal R^{-1}_{\mathcal S_{b_1(\epsilon)}}(P_{1})|>0$, there exists $x_{1}$ such that
$\rho(x_1, b_1(\epsilon)=h(x_{1},P_{1},b_1(\epsilon))
\leq \sqrt{2\sup_{D}|P|}\sqrt{\dfrac{\epsilon}{\kappa-1}}$.
Consequently,
$\kappa |P_{i}|-b_{i}(\epsilon)\leq \sqrt{2\epsilon(\kappa -1)\sup_{D}|P|}$
and therefore
\[
\rho(x,\kappa |P_{1}|-\epsilon)
=\max_{i}h(x,P_{i},b_{i}(\epsilon))
\leq \max_{i}\sqrt{2\sup_{D}|P|}\sqrt{\dfrac{\kappa |P_{i}|-b_{i}}{\kappa -1}}
\leq C\, \sqrt[4]{\epsilon}\to 0,
\]
as $\epsilon\to 0$.
Consequently, by continuity of $\rho(x,b_{1})$, given $X_{0}=|X_0|x_0\in Q_{r_{0}}$ with $0<|X_0|<\dfrac{\sigma}{\kappa-1}$,
there exists $b_{1}(X_{0})$ such that $\rho\left(x_0,b_{1}\right)=|X_{0}|$.

For the general case of a Radon measure $\mu$ in $\overline D$, 
we choose a sequence of measures $\mu_{\ell}$ such that each one is a finite combination of Dirac measures and $\mu_{\ell}\to \mu$ weakly with
$\mu_{\ell}(\overline{D})=\mu(\overline{D})$.
From the above, let $\mathcal S_{\ell}$ be the near field refractor corresponding to the measure $\mu_{\ell}$ and parameterized by
$\rho_{\ell}(x)x$ and passing through the point $X_{0}$.
Thus, $|X_0|\in \text{Range}(\rho_l)$ for all $l$.
From \eqref{eq:estimateofrhokappa>1}, $|X_0|<\lim_{b\to \alpha_P} h(x, P, b)$, where $\alpha_P=\kappa |P|-\dfrac{(\kappa-1)r_0^2}{2\sup_{D}|P|}$.
By Lemma \ref{lm:estimateofovalsk>1} and Lemma \ref{lm:weakcompactnesskappa>1}, one can apply Theorem \ref{thm:general measure} to obtain the existence of solutions.
\end{proof}


\setcounter{equation}{0}
\section{Further applications}\label{sec:furtherapplications}
To illustrate the general framework described in Section \ref{sec:abstract setup}, we briefly show how to 
recover the results for the far field refractor, proved using mass transport in \cite{gutierrez-huang:farfieldrefractor}, and also
the solution to the second boundary value problem for the Monge-Amp\`ere equation. 
We only state the results when the measure $\mu$ is a finite combination of Dirac measures. The general case for a general Radon measure follows by approximation as in Theorems \ref{thm:uniquenessabstractcase} and \ref{thm:existencegeneralcasekappa<1},
and noticing that the far field refractor problem is dilation invariant and
the second boundary value problem for Monge-Amp\`ere equation is translation invariant.

\subsection{Far field refractor, $\kappa<1$}
We have two domains $\Omega,\Omega^*\subset S^{n-1}$ satisfying the condition 
$x\cdot m\geq \kappa$ for all $x\in \Omega$, $m\in \Omega^*$, with $|\partial\Omega|=0$.
In this case, refractors are defined with supporting semi-ellipsoids $\rho(x,m,b)=\dfrac{b}{1-\kappa \,m\cdot x}$
in \cite[Definition 3.1]{gutierrez-huang:farfieldrefractor}.


%
We shall apply the setup in Section \ref{sec:abstract setup} with the spaces $X=\overline\Omega$, and $Y=\overline{\Omega^*}$.
The Radon measure $\omega$ in $\overline\Omega$ is now given by $\omega=fdx$ with 
$f\in L^1(\overline\Omega)$ nonnegative. 
If $\mathcal S$ is a far field refractor in the sense of  \cite[Definition 3.1]{gutierrez-huang:farfieldrefractor}, then it is proved in 
Lemma \ref{measurabilityrefractormapfarfield} below that the map $\Phi=\mathcal N_{\mathcal S}\in C_s(\overline\Omega, \overline{\Omega^*})$, where $\mathcal N_{\mathcal S}$ is defined in \cite[Definition 3.2]{gutierrez-huang:farfieldrefractor}.
From Lemma \ref{lm:radonmeasureontarget} we therefore obtain that
the set function
\begin{equation}\label{def:eqrefractormeasurefarfield}
\mathcal M_{\mathcal S, f}(F):=\int_{\mathcal N_{\mathcal S}^{-1}(F)}f\,dx,
\end{equation}
is a Radon measure defined on $\overline{\Omega^*}$. We call this measure 
{\it the far field refractor measure associated with $f$ and the refractor $\mathcal S$}.

We next introduce the family $\mathcal F$. Let $\mathcal S_\rho$ denote the far field refractor with defining radial function $\rho$ given by 
\cite[Definition 3.1]{gutierrez-huang:farfieldrefractor}.
We let $\mathcal F$ be the family of functions in $C^+(\overline\Omega)$ given by
$$\mathcal F=\{\rho(x): \mathcal S_\rho \text{ is a far field refractor}\}.$$ 
On $\mathcal F$ we define the mapping $\mathcal T$ by 
$$\mathcal T(\rho)=\mathcal N_{\mathcal S_\rho}.$$
To continue with the application of the results from Section \ref{sec:abstract setup}, 
we need to also show that $\mathcal T$ is continuous at each $\rho\in \mathcal F$ in the sense of Definition
\ref{def:definitionofmapTcontinuous}.
This is proved in Lemma \ref{lm:weakcompactnessfarfield} below.

\begin{lemma}\label{measurabilityrefractormapfarfield}
For each far field refractor $\mathcal S$, we have $\mathcal N_{\mathcal S}\in C_s(\overline\Omega, \overline \Omega^*)$.
\end{lemma}

\begin{proof}
Suppose $\mathcal S$ is parameterized by $\rho(x)$.
We first show that $\mathcal N_{\mathcal S}(\bar \Omega)=\bar \Omega^*$.  Because if $m\in \bar \Omega^*$, then letting 
\[
b_1=\inf\{b:\rho(x)\leq \rho(x,m,b)\text{ for all $x\in \bar \Omega$}\},
\]
we get that the semi ellipsoid $\rho(x,m,b_1)$ supports $\rho(x)$ at some $y\in \bar \Omega$.
Next show that $\mathcal N_{\mathcal S}(x)$ is single-valued for a.e. $x$ with respect to $\omega$.
Indeed, will prove that $\{x\in\bar \Omega: \mathcal N_{\mathcal S}(x) \text{ is not a singleton}\}\subset \{x\in \bar\Omega: \rho \text{ is not differentiable at $x$}\}$.
In fact, if $m_1,m_2\in \mathcal N_{\mathcal S}(x)$, and $\rho$ is differentiable at $x$, then 
$\rho$ has a unique supporting hyperplane $\Pi$ at $x$ having outer unit normal $\nu$.
Since the semi ellipsoids $\rho(x,m_1,b_1)$ and $\rho(x,m_2,b_2)$ support $\rho$ at $x$, the hyperplane $\Pi$ supports both semi ellipsoids at $x$, and then from the Snell law we get that
$x-\kappa \,m_1=\lambda_1\,\nu$ and $x-\kappa \,m_2=\lambda_2\,\nu$. But since $\lambda_1=\lambda_2=\Phi(x\cdot \nu)$, we get that $m_1=m_2$.
Since the graph of $\mathcal S$ is Lipschitz and $|\partial \Omega|=0$, the set of singular points of $\mathcal S$ has measure zero and therefore 
$\mathcal N_{\mathcal S}(x)$ is single-valued for a.e. $x\in \overline\Omega$.

To prove that $\mathcal N_{\mathcal S}$ is continuous, 
let $x_i\longrightarrow x_0$ and 
$m_i\in \mathcal N_{\mathcal S}(x_i)$.
Let $\rho(x,m_i,b_i)$ be a supporting semi ellipsoid to $\mathcal S$
at $\rho(x_i)x_i$. Then
\begin{equation}\label{eq:inequalityofO_{i}farfield}
\rho(x) \leq \dfrac{b_{i}}{1-\kappa\,m_i\cdot x} \qquad \text{ for }x\in \overline\Omega,
\end{equation}
with equality at $x=x_i$ and $x\cdot m_{i}\geq \kappa$ for all $x\in \overline{\Omega}$. Assume that $a_1\le \rho(x) \le a_2$
on $\overline\Omega$ for some constants $a_2\geq a_1>0$. 
From \eqref{eq:inequalityofO_{i}farfield} we then get
$a_{1}(1-\kappa)\leq b_{i}\leq a_2(1-\kappa^2)$.
Therefore selecting a subsequence we can assume that 
$m_i\longrightarrow m_0\in \overline \Omega^*$ and $b_i\longrightarrow b_0$,
as $i \longrightarrow \infty$.
By taking limit in \eqref{eq:inequalityofO_{i}farfield},
one obtains that semi ellipsoid $\rho(x,m_0,b_0)$ supports
$\mathcal S$ at $\rho(x_0)x_0$.
\end{proof}

\begin{lemma}\label{lm:weakcompactnessfarfield}
The far field refractor mapping $\mathcal T(\rho)=\mathcal N_{\mathcal S_\rho}$ is continuous at each $\rho\in \mathcal F$.
\end{lemma}
\begin{proof}
Suppose $\rho_j\longrightarrow \rho$ uniformly as $j\to \infty$. Let $x_0\in\overline\Omega$ and $m_j\in \mathcal N_{\mathcal S_{\rho_j}}(x_0)$. 
Then there exists $b_j$ such that 
$\rho_{j}(x)\leq \rho(x, m_{j},b_{j})$ for all $x\in \overline{\Omega}$ with equality 
at $x=x_{0}$ and with $x\cdot m_{j}\geq \kappa$. Selecting subsequences as in the proof of Lemma  \ref{measurabilityrefractormapfarfield}, we obtain $m_0\in \mathcal N_{\mathcal S_\rho}(x_0)$. 
\end{proof}

We therefore can apply Lemma \ref{lm:continuityofmeasuresabstractapproach} to obtain that the 
definition of refractor measure given in \eqref{def:eqrefractormeasurefarfield} is stable by uniform limits,
i.e., if $\rho_j\to \rho$ uniformly, then $\mathcal M_{\mathcal S_{\rho_j}, f}\to \mathcal M_{\mathcal S_{\rho}, f}$ weakly.

To be able to apply Theorem \ref{thm:abstractcasediscritecase}, we next need to verify that the family $\mathcal F$ and the map $\mathcal T$ satisfy conditions (A1)-(A3) from Section 
\ref{sec:abstract setup}.
Indeed, (A1) follows immediately from the definition of far field refractor.
Condition (A2) immediately follows from the definition of far field refractor.

It remains to verify (A3).
Indeed, with the notation in condition (A3) we will take 
\[
h_{t,y_0}(x)=\rho(x,m,b)
\]
with the understanding that $t=b$, and $y_0=m$, and 
$\rho(x,m,b)$ is the semi ellipsoid $E(m,b)$.
In other words, we will show that the family 
$$\left\{\rho(\cdot, m,b): m\in \bar \Omega^*, 0<b<+\infty \right\} \subset \mathcal F,$$
and verifies (A3).
Indeed, it is clear that $\rho(\cdot,m,b)$ is a far field refractor, and in particular, 
$m\in \mathcal T(\rho(\cdot, m,b))(x)$ for all $x\in \bar \Omega$, that is, (A3)(a) holds. 
%
Condition (A3)(b) is trivial. Condition (A3)(c) follows from $\rho(x,m,b)\leq \dfrac{b}{1-\kappa^2}$.
Finally, (A3)(d) follows from $|\rho(x,m,b')-\rho(x,m,b)|\leq \dfrac{|b'-b|}{1-\kappa^2}$.

%

The notion of weak solution is introduced through conservation
of energy.

\begin{definition}\label{def:definitionoffarfieldrefractorkappa<1}
A far field refractor $\mathcal{S}$ is a weak solution of the far field refractor problem for the case $\kappa<1$ with emitting illumination intensity
$f(x)$ on $\overline\Omega$ and prescribed refracted illumination
intensity $\mu$ on $\overline{D}$ if for any Borel set
$F\subset \overline{D}$
\begin{equation}
\mathcal M_{\mathcal S, f}(F)
=\int_{\mathcal N_{\mathcal S}^{-1}(F)}f\,dx=\mu(F).
\end{equation}
\end{definition}

We are now ready to apply Theorem \ref{thm:abstractcasediscritecase} to solve the far field refractor problem when the measure $\mu$ is a linear combination of deltas.

\begin{theorem}\label{thm:existencesumofdiracmassesfarfield}
Let $m_{1},\cdots , m_{N}$ be distinct points in $\overline{\Omega^*}$,  
$g_{1},\cdots , g_{N}$ are positive numbers, $f\in L^{1}(\Omega)$ such that $x\cdot m_i\geq \kappa$ for $x\in \Omega$, $1\leq i\leq N$, and
\begin{equation}\label{eq:conservationofenergykappa<1farfield}
\int_{\overline{\Omega}} f(x)\,dx=\sum_{i=1}^{N} g_{i}.
\end{equation}
Then 
there exist positive numbers $b_1,b_{2},\cdots , b_{N}$ such that 
$
\mathcal S= \{\rho(x)x: x\in \overline{\Omega}\}
$ with
\begin{equation}\label{eq:definitionoffarfieldpoliellipsoid}
\rho(x)=\min_{1\leq i \leq N} \rho(x,m_{i},b_{i})
\end{equation}
is a weak solution to the far field refractor problem.
Moreover, $\mathcal M_{\mathcal S, f}\left(\{m_{i}\}\right)=g_{i}$ for $1\leq i \leq N$.
\end{theorem}
\begin{proof}
To prove the theorem, we apply Theorem \ref{thm:abstractcasediscritecase}. So we only need to verify that there exists 
$\rho_0(x)=\min_{1\leq i \leq N} \rho(x,m_{i},b^0_{i})$ satisfying
$\mathcal M_{\mathcal S_{\rho_0},f}(m_i)\le g_i$, $2\leq i\leq N$.
From Remark \ref{rmk:existenceofrho0forb1close} this follows by choosing $b_1$ close to zero.

\end{proof}

\subsection{Far field case, $\kappa>1$}\label{subsect:farfieldkappa>1}

Using the results from Subsection \ref{subsect:convexcase} and adapting the above arguments we can easily deal with the case $\kappa>1$.
In this case, we assume $x\cdot m\geq \dfrac{1}{\kappa}+\delta$, $|\partial\Omega|=0$, and the definition of far field refractor is made with 
supporting semi-hyperboloids $\rho(x,m,b)=\dfrac{b}{\kappa \,m\cdot x-1}$ as in \cite[Definition 4.1]{gutierrez-huang:farfieldrefractor}.
The far field refractor mapping of $\mathcal S$ is given by \cite[Definition 4.2]{gutierrez-huang:farfieldrefractor}.
The family $\mathcal F$ is then given by $\mathcal F=\{\rho:\text{$S_\rho$ is a far field refractor for $\kappa>1$}\}$.
We have 
$$
\left\{\rho(\cdot, m,b): m\in \bar \Omega^*, 0<b<+\infty \right\} \subset \mathcal F,
$$ 
where the functions $\rho(\cdot, m,b)$ now satisfying conditions (A1')-(A2') and (A3'') from Subsection \ref{subsect:convexcase}.
Weak solutions of the far field refractor problem for $\kappa>1$ are defined as in Definition \ref{def:definitionoffarfieldrefractorkappa<1}.

Therefore, applying Theorem \ref{thm:abstractcasediscritecaseconvexcase}, 
we obtain the following theorem.

\begin{theorem}\label{thm:existencesumofdiracmassesfarfieldkappa>1}
Let $m_{1},\cdots , m_{N}$ be distinct points in $\overline{\Omega^*}$,  
$g_{1},\cdots , g_{N}$ are positive numbers, $f\in L^{1}(\Omega)$ such that $x\cdot m_i\geq \dfrac{1}{\kappa}+\delta$ for $x\in \Omega$, $1\leq i\leq N$, and
\begin{equation}\label{eq:conservationofenergykappa>1farfield}
\int_{\overline{\Omega}} f(x)\,dx=\sum_{i=1}^{N} g_{i}.
\end{equation}
Then 
there exist positive numbers $b_1,b_{2},\cdots , b_{N}$ such that 
$
\mathcal S= \{\rho(x)x: x\in \overline{\Omega}\}
$ with
\begin{equation}\label{eq:definitionoffarfieldpoliellipsoid}
\rho(x)=\max_{1\leq i \leq N} \rho(x,m_{i},b_{i})
\end{equation}
is a weak solution to the far field refractor problem for $\kappa>1$.
Moreover, $\mathcal M_{\mathcal S, f}\left(\{m_{i}\}\right)=g_{i}$ for $1\leq i \leq N$.
\end{theorem}

\subsection{The second boundary value problem for the Monge-Amp\`ere equation}

We assume here that $X=\bar \Omega$ and $Y=\bar \Omega^*$ with $\Omega,\Omega^*$ bounded convex domains in $\R^n$.
We have a Radon measure $\omega$ in $\bar \Omega$ given by $w=f\,dx$ with $f\in L^1(\Omega)$, $f$ nonnegative.
The subdifferential of the function $u:\overline \Omega\to \R$ is given by
\[
\partial u(x_0)=\{p\in \bar \Omega^*: u(x)\geq u(x_0)+p\cdot (x-x_0) \text{ for all $x\in \overline \Omega$}\}.
\]
We let the family $\mathcal F=\{u\in C(\bar \Omega): u \text{ is convex and $\partial u(y)\cap \overline{\Omega^*}\neq \emptyset$ $\forall y\in \overline \Omega$}\}$.
On $\mathcal F$ we define the mapping $\mathcal T$ by
$
\mathcal T(u)=\partial u.
$
One needs to prove that $\partial u\in C_s(\bar \Omega,\bar \Omega^*)$ and the map
$\mathcal T(u)=\partial u$ is continuous at each $u\in \mathcal F$.
To show $\partial u(\bar \Omega)=\bar \Omega^*$, we proceed exactly as at the beginning of the proof of
Lemma \ref{measurabilityrefractormapfarfield}. Everything else follows from well known properties of the subdifferential,
see \cite{Gut:book}.

The family $\mathcal F$ satisfies (A1') and (A2'), and we verify that it also satisfies (A3''), all from Subsection \ref{subsect:convexcase}.
We let 
$
h_{t,y_0}(x)=x\cdot p+b,
$
where $t=b$ and $y_0=p$. Here $x\in \bar \Omega$ and $p\in \bar \Omega^*$.
We show that 
\[
\{x\cdot p+b: p\in \bar \Omega^*, -\infty<b<\infty\}\subset \mathcal F
\]
satisfies condition (A3'').
In fact, $p\in \partial (x\cdot p+b) (x)$ for all $x\in \bar \Omega$, so (A3'')(a) holds.
(A3'')(b) is trivial. Since $x\cdot p+b\geq b-|x|\,|p|\geq b-C$, where $C$ depends on the diameters of $\Omega$ and $\Omega^*$, (A3'')(c) follows.
(A3'')(d) trivially holds.

We are now ready to apply Theorem \ref{thm:abstractcasediscritecaseconvexcase} 
to solve the second boundary value problem for the Monge-Amp\`ere equation when the measure $\mu$ is a linear combination of deltas.
That is, to find a convex function $u\in \mathcal F$ such that $\partial u(\bar \Omega)=\bar \Omega^*$ and solving
$\int_{(\partial u)^{-1}(E)}f(x)\,dx=\mu(E)$ for each Borel set $E\subset \bar \Omega^*$.

\begin{theorem}\label{thm:existencesumofdiracmassessecondbdryMA}
Let $p_{1},\cdots , p_{N}$ be distinct points in $\overline{\Omega^*}$,  
$g_{1},\cdots , g_{N}$ are positive numbers, $f\in L^{1}(\Omega)$,
 and
\begin{equation}\label{eq:conservationofenergysecondbdryMA}
\int_{\overline{\Omega}} f(x)\,dx=\sum_{i=1}^{N} g_{i}.
\end{equation}
Then 
there exist positive numbers $b_1,b_{2},\cdots , b_{N}$ such that the convex function
\begin{equation}\label{eq:solutionsecondbdryMA}
u(x)=\max_{1\leq i \leq N}\{ x\cdot p_i+b_i\}
\end{equation}
solves the second boundary value problem for the Monge-Amp\`ere equation.
\end{theorem}

\section{Appendix: Derivation of the pde}\label{sec:derivationofpde}
\footnotesize
Suppose the function defining the refractor is $\rho(x_{1},\cdots , x_{n-1},x_{n})$ and
set $x'=(x_{1},\cdots , x_{n-1})$. We have points $(x',x_{n})\in \Omega\subset S^{n-1}$,
so we think of the region $\Omega$ defined by $\{(x',\sqrt{1-|x'|^{2}}):x'\in U\}$ and therefore we identify $\Omega$ with $U$. We also think of the defining function 
$\rho$ as  a function $\rho=\rho(x')$ with $x'\in U$.
For the derivation of the equation we assume that $\rho$ is $C^2$.
Following the paper \cite{karakhanyan-wang:nearfieldreflector}, we use the notation 
$D\rho=(\partial_{1}\rho,\cdots ,\partial_{n-1}\rho)$ and $\hat D\rho=
(D\rho,0)$. We also use the notation $x=(x_{1},\cdots ,x_{n})\in S^{n-1}$ and
let $y\in S^{n-1}$ be the refracted direction of the ray $x$ by the surface $\rho(x)x$, that is,
\begin{equation}\label{eq:formulaforY}
y=\dfrac{1}{\kappa}\left( x- \Phi(x\cdot \nu)\nu\right),
\end{equation}
where $\nu$ is the outer unit normal to the refractor at the point $\rho(x)x$, and
$\Phi(t)=t-\kappa \sqrt{1-\kappa^{-2}(1-t^{2})}$.

\subsection{Case when the target domain $D\subset \{x_{n}=0\}$.}\label{subsec:pdewhenDiscontainedinahyperplane}
Notice that this is compatible with hypotheses H1 and H2 if $\kappa<1$ or H3 and H4 if $\kappa>1$, when $\Omega$ is
above or near the hyperplane $x_n=0$.
Suppose the surface refracts off the ray with direction $x$ into the point $Z\in D$.
Then
\[
Z=\rho(x)x+ |Z-\rho(x)x|y.
\]
We denote by $T$ the map $x\mapsto Z$ and we regard it defined in $U$, that is,
$T:U\to D$, where $D$ is the target screen.
Since $D\subset \{x_n=0\}$, we have that $T(x')=(z_1,\cdots ,z_{n-1},0)$, and 
the Jacobian of $T$ is then the matrix
$DZ=\left(\partial_j z_i \right)_{ij}$, $1\leq i,j\leq n-1$.
If $d S_\Omega$ and $dS_D$ denote the surface area elements in $\Omega$ and in $D$, respectively, then $\det DZ=\dfrac{dS_D}{dS_\Omega}$.
Noticing that $dS_\Omega=\dfrac1{\sqrt{1-|x'|^2}}dS_U$, and since $f$ and $g$ are the energy distributions in $\Omega$ and $D$, respectively, we obtain the equation
\begin{equation}\label{eq:firstpde}
\det \,DZ=\dfrac{f}{g\sqrt{1-|x'|^2}}.
\end{equation}
We now find the explicit form of $DZ$ which will yield the pde satisfied by $\rho$.

From \cite[Formula (2.15)]{karakhanyan-wang:nearfieldreflector} we have the following expression for the outer normal (the change in the sign is due to the 
direction of the normal):
\begin{equation}\label{eq:formulafornormal}
\nu=\dfrac{-\hat D\rho+x\left(\rho(x')+D\rho(x')\cdot x'\right)}{\sqrt{\rho^{2}+|D\rho|^{2}-(D\rho\cdot x')^{2}}},
\end{equation}
and so
\[
x\cdot \nu=
\dfrac{\rho}{\sqrt{\rho^{2}+|D\rho|^{2}-(D\rho\cdot x')^{2}}}.
\]
We now calculate $d=|Z-\rho(x')x|$. Since $D\subset \{x_{n}=0\}$, we have
$0=\rho(x')x_{n}+d y_{n}$ and so $d=-\rho(x')\dfrac{x_{n}}{y_{n}}$.
Also 
\begin{align*}
y_{n}&=\dfrac{1}{\kappa}\left( x_{n}-\Phi(x\cdot \nu)\nu_{n}\right)\\
&=\dfrac{1}{\kappa}\left( x_{n}-\Phi(x\cdot \nu)\dfrac{x_{n}\left(\rho+D\rho\cdot x'\right)}{\sqrt{\rho^{2}+|D\rho|^{2}-(D\rho\cdot x')^{2}}}\right)\\
&=
x_{n}
\dfrac{1}{\kappa}\left(1- \Phi(x\cdot \nu)\dfrac{\rho+D\rho\cdot x'}{\sqrt{\rho^{2}+|D\rho|^{2}-(D\rho\cdot x')^{2}}}\right)\\
&=x_{n}\,
\dfrac{A}{\kappa},
\end{align*}
so
\[
d=-\rho \, \dfrac{\kappa}{A}.
\]
From \eqref{eq:formulaforY} we then have
\[
y=\dfrac{1}{\kappa}\left( A\, x +\Phi(x\cdot \nu) \dfrac{\hat D\rho}{\sqrt{\rho^{2}+|D\rho|^{2}-(D\rho\cdot x')^{2}}}\right). 
\]
Therefore
\begin{align*}
Z&=\rho(x)x-\rho\dfrac{\kappa}{A} y
=
-\dfrac{\rho}{A}\dfrac{\Phi(x\cdot \nu)}{\sqrt{\rho^{2}+|D\rho|^{2}-(D\rho \cdot x')^{2}}}\hat D\rho\\
&=
\rho \dfrac{\Phi(x\cdot \nu)}{-G + \Phi(x\cdot \nu) \left(\rho+D\rho\cdot x'\right)}\hat D\rho\\
&=
F(x',\rho(x'),D\rho(x'))\hat D (\rho^2),
\end{align*}
where 
\[
F(x',\rho(x'),D\rho(x'))=
\dfrac12  \dfrac{\Phi(x\cdot \nu)}{-G + \Phi(x\cdot \nu) \left(\rho+D\rho\cdot x'\right)}
\]
with $G=\sqrt{\rho^{2}+|D\rho|^{2}-(D\rho \cdot x')^{2}}$.
It is convenient to use the notation
\[
F(x',u,p)
=\dfrac12
\dfrac{\Phi \left( \dfrac{u}{\sqrt{u^2+|p|^2-(p\cdot x')^2}} \right)}{-\sqrt{u^2+|p|^2-(p\cdot x')^2}+
(u+p\cdot x')\,\Phi\left( \dfrac{u}{\sqrt{u^2+|p|^2-(p\cdot x')^2}} \right)}.
\]
So $z_i=F(x',\rho(x'),D\rho(x')) (\rho^2)_{x_i}$ for $1\leq i\leq n-1$ and $z_n=0$. Differentiating this with respect 
to $x_j$  we get
\begin{align*}
\partial_j z_i
&=
F\,(\rho^2)_{x_ix_j}
+
(\rho^2)_{x_i}\left( F_{x_j}+F_u \,\rho_{x_j}+\sum_{k=1}^{n-1} F_{p_k} \, \rho_{x_kx_j}\right)\\
&=
2 F\, \rho \, \rho_{x_i x_j}
+
2 F\, \rho_{x_i} \, \rho_{x_j}+
(\rho^2)_{x_i}\left( F_{x_j}+F_u \,\rho_{x_j}+\sum_{k=1}^{n-1} F_{p_k} \, \rho_{x_kx_j}\right).
\end{align*}
If $\eta,\xi$ are row vectors in $\mathbf{R}^n$, the tensor product is the $n\times n$ matrix defined by
\[
\xi \otimes \eta=\xi^t \eta,
\]
with the multiplication of matrices. 
Then the Jacobian matrix $DZ=(\partial_j z_i)$ can be written as
\begin{align*}
DZ=
2 \rho \,F \,D^2\rho+2 F D\rho\otimes D\rho +
D(\rho^2) \otimes D_{x'} F
 +
F_u \,D(\rho^2) \otimes D\rho
+
D(\rho^2) \otimes \left(D_pF\,D^2\rho \right).
\end{align*}
If $\xi, \eta$ are row vectors and $A$ is an $n\times n$ matrix, then
$\xi \otimes (\eta A)=(\xi\otimes \eta)A$, and we then obtain the formula
\begin{align*}
DZ
=2 \rho \left\{ F I+  D\rho \otimes D_pF \right\} D^2\rho 
+
2 F D\rho\otimes D\rho +
D(\rho^2) \otimes D_{x'}F
+
F_u \,D(\rho^2)\otimes D\rho.
\end{align*}
We have from the Sherman-Morrison formula
that if $\mathcal M= I+\xi\otimes \eta$, with $\xi$ and $\eta$ row vectors, then,
\begin{equation}\label{eq:formulaformatrices}
\det \,\mathcal M=1+\xi\cdot \eta,
\text{ and }
\mathcal M^{-1}=I-\dfrac{\xi\otimes \eta}{1+\xi\cdot \eta}.
\end{equation}
Therefore, if 
\begin{equation}\label{eq:definitionofmathcalM}
 \mathcal M=I+ \dfrac{1}{ F(x',\rho,D\rho)}D\rho\otimes D_pF ,
 \end{equation}
then
\begin{equation}\label{eq:defofM}
\mathcal M^{-1}
=
I-\dfrac{D\rho\otimes D_pF}{F(x,\rho,D\rho)+D\rho \cdot D_pF}.
\end{equation}
Hence
\begin{align*}\label{eq:formulaforDz}
DZ=
2\rho F \mathcal M D^2\rho
+
2 F D\rho\otimes D\rho +
D(\rho^2)\otimes D_{x'}F 
+
F_u \,D(\rho^2)\otimes D\rho 
=2\rho F\mathcal M\,D^2\rho
+B,
\end{align*}
and so
\[
\dfrac{1}{2\rho F} \,\mathcal M^{-1} \,DZ=D^2\rho+\dfrac{1}{2\rho F}
\,\mathcal M^{-1}\,B.
\]
We have $\det \mathcal M^{-1}=\dfrac{F}{F+D\rho\cdot D_pF}$.
Therefore
\[
\det DZ=\det\left(D^2\rho+\dfrac{1}{2\rho F}
\,\mathcal M^{-1}\,B\right)\,(2\rho)^{n-1} F^{n-2}
\left(F+D\rho \cdot D_pF \right).
\]
Since $(\alpha \otimes \beta)(\xi\otimes \eta)
=(\beta\cdot \xi)(\alpha\otimes \eta)$, from the form of $\mathcal M^{-1}$ and $B$ we have that
\[
\mathcal M^{-1}B
=
\dfrac{2F}{F+D\rho \cdot D_p F}\left[ (F+\rho F_u) D\rho\otimes D\rho
+
\rho D\rho\otimes D_{x'}F \right].
\]
Combining this with \eqref{eq:firstpde}, we obtain that $\rho$ satisfy the following
pde of Monge-Amp\`ere type:
\begin{align}\label{eq:pdeforrho}
\det\left(D^2\rho+\mathcal A(x',\rho,D\rho)\right) =
\dfrac{f}{g\sqrt{1-|x'|^2}\,(2\rho)^{n-1} F^{n-2}
\left(F+D\rho \cdot D_pF \right)},
\end{align}
where 
\[
\mathcal A(x',\rho,D\rho)=
\dfrac{1}{\rho (F+D\rho \cdot D_p F)} \left[ (F+\rho F_u) D\rho\otimes D\rho
+
\rho  D\rho \otimes D_{x'}F \right].
\]

\subsection{Case when $D$ is contained in a hypersurface}
We assume the target domain $D$ is contained in a hypersurface $\{P\in \R^n:\psi(P)=0\}$, where $\psi$ is $C^1$ and
$\psi_{P_n}\neq 0$. We recall that we assume the configuration given by either (H1) and (H2), if $\kappa<1$, or
(H3) and (H4), if $\kappa>1$.
In order to find the pde in the general case we will use the calculations from Subsection \ref{subsec:pdewhenDiscontainedinahyperplane}.

Given the direction $x\in \Omega$,
the point $\rho(x)x$ is refracted with direction $y$ into the point $Z\in D$. 
The normal at $Z=(z_1,\cdots ,z_n)\in D$ is given by 
$D\psi$ and so the map $T:x\mapsto Z$ has Jacobian matrix $J$ satisfying
\[
\det J=\dfrac1{|D\psi|} \det \left[ 
\begin{matrix}
\partial_1 z_1 & \cdots & \partial_{n-1}z_1 & \psi_{P_1} \\
\partial_1 z_2 & \cdots & \partial_{n-1}z_2 & \psi_{P_2} \\
\vdots & \vdots & \vdots & \vdots\\
\partial_1 z_n & \cdots & \partial_{n-1}z_n & \psi_{P_n} \\
\end{matrix}
\right].
\]
We have $Z=Z(x')=(z_1(x_1,\cdots ,x_{n-1}),\cdots ,
z_n(x_1,\cdots ,x_{n-1}))$ and $\psi \left( Z(x_1,\cdots ,x_{n-1})\right)=0$,
so differentiating with respect to $x_j$ we get
$\sum_{i=1}^{n} \psi_{P_i} \partial_{x_j}z_i=0$ and so
\[
\partial_{x_j}z_n=-\dfrac{1}{\psi_{P_n}} \sum_{i=1}^{n-1} \psi_{P_i} \partial_{x_j}z_i,
\qquad
\text{ for $j=1,\cdots ,n-1$.} 
\]
Inserting these into $\det J$ yields
\begin{equation}\label{eq:Jacobiangeneralcase}
\det J=\dfrac{|D\psi|}{\psi_{P_n}}\det \left[ 
\begin{matrix}
\partial_1 z_1 & \cdots & \partial_{n-1}z_1 \\
\partial_1 z_2 & \cdots & \partial_{n-1}z_2 \\
\vdots & \vdots & \vdots \\
\partial_1 z_{n-1} & \cdots & \partial_{n-1}z_{n-1} \\
\end{matrix}
\right].
\end{equation}
For $Z\in D$, $x\in \Omega$, with $Tx=Z$, let $y=\dfrac{Z-\rho(x)x}{|Z-\rho(x)x|}$,
and let $\ell$ be the line with direction $y$ passing through the point $\rho(x)x$.
Let $W$ be the intersection point between $\ell$ and the hyperplane $x_n=0$.
We write 
\[
W=\rho(x)x+d_0 y.
\] 
The existence of $W$ follows from (H1) and (H2) if $e_n\in \Omega$ and $\tau<\kappa$.
Indeed, we notice that $\ell$ does not intersect $x_n=0$ if and only if
$\ell$ is perpendicular to $e_n$, or equivalently the
vector $y$  is perpendicular to $e_n$.
From (H1), $Z/|Z|$ is contained in the solid cone with axis $e_n$ with opening $\kappa+\tau$ for $Z\in D$,
i.e. $Z/|Z|\cdot e_n\geq \kappa+\tau$. On the other hand, since the refractor is contained in $Q_{r_0}$, we have
$
A=\left| \dfrac{Z}{|Z|}-y\right|
\leq \dfrac{2\rho(x)}{|Z-\rho(x)x|}\leq \dfrac{2r_0}{|Z-\rho(x)x|}.$
Now from (H2)
\[
|Z-\rho(x)x|\geq |Z|-r_0\geq \text{dist }(D,0)-r_0\geq 
r_0\left(\dfrac{1+\kappa -\tau}{\tau}\right),
\]
so
$
A\leq \dfrac{2\tau}{1+\kappa -\tau}$.
Therefore
\[
y\cdot e_n=
\left(y-\dfrac{Z}{|Z|}\right)\cdot e_n +\dfrac{Z}{|Z|}\cdot e_n
\geq 
-\dfrac{2\tau}{1+\kappa -\tau}+\kappa +\tau=(\kappa-\tau)\dfrac{1+\kappa +\tau}{1+\kappa -\tau}:=r.
\]
Thus, if $\tau<\kappa$, then $r>0$,
and so the vector $y$ lies in a cone with axis $e_n$ that does not intersect the plane $x_n=0$ which proves the existence of $W$.
Similarly, if $\kappa>1$, then the existence of $W$ also follows from (H3) and (H4) assuming again $e_n\in \Omega$ and that $\tau$ is sufficiently small.

A calculation as in the case when $D\subset \{x_n=0\}$ yields that
\[
W=F\left(x',\rho(x'),D\rho(x')\right) \hat D(\rho^2).
\]
If we write 
\[
Z=\rho(x')x+ t(x')\left( W-\rho(x')x\right),
\]
then $Z-\rho(x')x=t(x')\left( W-\rho(x')x\right)$ and making the dot product of
this equation with $e_n$ we obtain that
\[
t(x')=\dfrac{\rho(x')x_n - z_n}{\rho(x')x_n},
\]
where $x_n=\sqrt{1-|x'|^2}$.
In view of \eqref{eq:Jacobiangeneralcase}, we only need to calculate
$\partial_j z_i$ for $1\leq i,j\leq n-1$.
Set
$z'=(z_1,\cdots ,z_{n-1})$ and $w'=(w_1,\cdots ,w_{n-1})$, we get
\begin{equation*}
\partial_j z_i
=
\partial_j t\left( w_i-\rho(x')x_i\right)+(1-t)\partial_j\left( \rho(x')x_i\right)
+
t \partial_j w_i,
\qquad i,j=1,\cdots ,n-1.
\end{equation*}
This equation written in matrix form is
\begin{equation}\label{eq:derivativeofzwitht}
Dz'=(w' -\rho x')\otimes Dt + (1-t) D(\rho x') + t Dw'.
\end{equation}
We calculate $\partial_j t$.
Differentiating the equation $\psi\left( t W + (1-t) \rho(x')x \right)=0$ with respect to 
$x_j$, $1\leq j\leq n-1$, we get
\begin{align*}
\partial_j t\, D\psi \cdot (W- \rho x)
&=
-(1-t)\sum_{i=1}^n \psi_{P_i}
\partial_j\left( \rho x_i\right)
-
t\sum_{i=1}^n \psi_{P_i}
\partial_jw_i\\
&=
- 
\sum_{i=1}^{n-1} \psi_{P_i}
\left( (1-t) \partial_j\left( \rho x_i\right) + t \partial_jw_i\right)
-
(1-t) \psi_{P_n}
\partial_j\left( \rho x_n\right),
\end{align*}
since $w_n=0$.
So we have the formula for the vector $Dt=(\partial_1t,\cdots, \partial_{n-1}t)$
\begin{equation}\label{eq:formulaforDt}
Dt=
-\beta \tilde D \psi
 \left( (1-t) D(\rho  x') + t Dw' \right) -\beta
(1-t)  \psi_{P_n} D(\rho  x_n),
\end{equation}
where $\beta=\dfrac{1}{D\psi \cdot (W- \rho x)}$, and we used the notation
$\tilde D \psi=(\psi_{P_1},\cdots, \psi_{P_{n-1}})$.
For simplicity in the calculation let $A=(1-t) D(\rho  x') + t Dw'$.
So from  
\eqref{eq:derivativeofzwitht}
\[
Dz'=(w' -\rho x')\otimes Dt + A,
\]
and from \eqref{eq:formulaforDt}
\[
Dt=
-\beta\tilde D \psi A-\beta
(1-t)  \psi_{P_n} D(\rho x_n).
\]
So inserting \eqref{eq:formulaforDt} into \eqref{eq:derivativeofzwitht}
and using the formula $\xi\otimes(\eta A)=(\xi\otimes \eta)A$\footnote{$\xi\otimes \eta=\xi^t \eta$ as multiplication of matrices.}, 
we obtain
\begin{align*}
Dz'&=
\left( I -\beta (w' -\rho x') \otimes \tilde D\psi\right)A
-\beta
(1-t)  \psi_{P_n} (w' -\rho x')\otimes D(\rho  x_n).
\end{align*}
Letting $\mathcal B=  I -\beta (w' -\rho x')\otimes \tilde D\psi$,
from \eqref{eq:formulaformatrices} we get
\[
\mathcal B^{-1}
=
I + \dfrac{\beta  (w' -\rho x')\otimes \tilde D\psi}{1-\beta (w' -\rho x')\cdot \tilde D \psi }.
\]
Since $\dfrac{1}{\beta}
=D\psi \cdot (W-\rho x)
=
\tilde D\psi \cdot (w'-\rho x')-\psi_{P_n} \rho x_n$, it follows that
\[
\mathcal B^{-1}
=
I - \dfrac{(w' -\rho x')\otimes \tilde D\psi}{\psi_{P_n} \rho x_n}.
\]
Therefore
\begin{align*}
Dz'&=\mathcal B \left(A
-\beta
(1-t)  \psi_{P_n} \mathcal B^{-1}\left( (w'-\rho x')\otimes D(\rho x_n) \right)\right).
\end{align*}
We have
\begin{align*}
\mathcal B^{-1}\left( (w'-\rho x')\otimes D(\rho x_n)\right)
&=\left( I - \dfrac{(w' -\rho x')\otimes \tilde D\psi}{\psi_{P_n} \rho x_n}\right)\left((w'-\rho x') \otimes D(\rho x_n)\right)\\
&=
(w'-\rho x')\otimes D(\rho x_n)\\
&\qquad -
\dfrac{1}{\psi_{P_n} \rho x_n} \left((w' -\rho x') \otimes \tilde D\psi \right)
\left( (w'-\rho x')\otimes D(\rho x_n)\right)\\
&=
(w'-\rho x')\otimes D(\rho x_n)\\
&\qquad -
\dfrac{1}{\psi_{P_n} \rho x_n}
\left( \tilde D\psi\cdot (w' -\rho x') \right)
(w'-\rho x')\otimes D(\rho x_n)\\
&=
(w'-\rho x')\otimes D(\rho x_n)\\
&\qquad -
\dfrac{1}{\psi_{P_n} \rho x_n}
\left( \dfrac{1}{\beta} + \psi_{P_n} \rho x_n \right)
(w'-\rho x') \otimes D(\rho x_n)\\
&=
 -
\dfrac{1}{\beta \psi_{P_n} \rho x_n}
(w'-\rho x')\otimes D(\rho x_n).
\end{align*}
Consequently
\begin{align*}
Dz'&=\mathcal B\left(A
+
(1-t)  \dfrac{1}{\rho x_n} (w'-\rho x')\otimes D(\rho x_n)\right)\\
&=\mathcal B
\left(t Dw' + (1-t) D(\rho x')  
+
(1-t)  \dfrac{1}{\rho x_n} (w'-\rho x')\otimes D(\rho x_n)\right).\end{align*}
We recall from \eqref{eq:formulaforDz} that 
\[
Dw'
=2\rho F(x',\rho,D\rho)\,\mathcal M \,D^2\rho +B
\] 
with $\mathcal M$ from \eqref{eq:definitionofmathcalM}, so
\begin{align*}
Dz'&=\mathcal B\left(t  2\rho F\mathcal M \,D^2\rho+t B + (1-t) D(\rho x')  
+
(1-t)  \dfrac{1}{\rho x_n}  (w'-\rho x') \otimes D(\rho x_n)\right)\\
&=
\mathcal B 2t\rho F \mathcal M 
\left(D^2\rho +\dfrac{1}{2t\rho F} \mathcal M^{-1} \left(
t B + (1-t)\left( D(\rho x')  
+
  \dfrac{1}{\rho x_n} (w'-\rho x') \otimes D(\rho x_n)\right)\right)\right).
\end{align*}
We have $\det \mathcal B= 1-\beta (w'-\rho x')\cdot \tilde D \psi =-\beta \psi_{P_n} \rho x_n$ and $\det \mathcal M=1+\dfrac{1}{F}(D\rho\cdot D_pF)$.
So 
\[
\det Dz'=
(2t\rho F)^{n-1}(-\beta \psi_{P_n} \rho x_n)
\left(1+\dfrac{1}{F}(D\rho\cdot D_pF) \right)
\det \left(D^2\rho +\mathcal A \right),
\]
with 
\[
\mathcal A=\mathcal A(x', \rho, D\rho)
=
\dfrac{1}{2t\rho F} \mathcal M^{-1} \left(
t B + (1-t)\left( D(\rho x')  
+
  \dfrac{1}{\rho x_n} (w'-\rho x') \otimes D(\rho x_n)\right)\right).
  \]
From \eqref{eq:firstpde} and \eqref{eq:Jacobiangeneralcase} we then obtain that $\rho$ satisfy the pde of Monge-Amp\`ere type
\begin{equation}\label{eq:monge-amperegeneralcase}
\det \left(D^2\rho +\mathcal A \right)
=
\dfrac{f}{g(1-|x'|^2) |D\psi|( 2t )^{n-1}\rho^{n}(-\beta )F^{n-2}(F+ D\rho\cdot D_p F)}.
\end{equation}

\normalsize
\providecommand{\bysame}{\leavevmode\hbox to3em{\hrulefill}\thinspace}
\providecommand{\MR}{\relax\ifhmode\unskip\space\fi MR }
\providecommand{\MRhref}[2]{%
  \href{http://www.ams.org/mathscinet-getitem?mr=#1}{#2}
}
\providecommand{\href}[2]{#2}

\end{document}